\documentclass[11pt]{amsart}
\usepackage[T1]{fontenc}
\usepackage{lmodern,comment}
\usepackage{amsmath,amssymb,mathtools}
\usepackage{booktabs}
\usepackage[margin=1.1in]{geometry}
\usepackage{microtype}
\usepackage[colorlinks=true,linkcolor=blue,citecolor=blue,urlcolor=blue]{hyperref}
\usepackage[dvipsnames]{xcolor}

\definecolor{FaridColor}{RGB}{166,76,0}

\newtheorem{theorem}{Theorem}[section]
\newtheorem{proposition}[theorem]{Proposition}
\newtheorem{lemma}[theorem]{Lemma}
\newtheorem{corollary}[theorem]{Corollary}
\theoremstyle{definition}
\newtheorem{definition}[theorem]{Definition}

\newtheorem{question}[theorem]{Question}
\theoremstyle{remark}
\newtheorem{remark}[theorem]{Remark}
\theoremstyle{plain}

\newtheorem{maintheorem}{Theorem}

\newtheorem{mainproposition}[maintheorem]{Proposition}

\numberwithin{equation}{section}

\newcommand{\kk}{\Bbbk}
\newcommand{\NCSym}{\operatorname{NCSym}}
\newcommand{\NCQSym}{\operatorname{NCQSym}}
\newcommand{\FQSym}{\operatorname{FQSym}}
\newcommand{\Sym}{\operatorname{Sym}}
\newcommand{\SC}{\operatorname{SC}}
\newcommand{\SI}{\operatorname{SI}}

\newcommand{\supp}{\operatorname{supp}}

\newcommand{\Hilb}{\operatorname{Hilb}}
\newcommand{\tr}{\operatorname{tr}}

\newcommand{\Span}{\operatorname{span}}
\newcommand{\std}{\operatorname{std}}
\newcommand{\inv}[1]{\left\langle #1\right\rangle}

\title{Finite noncommutative coinvariant algebras}
\author{Farid Aliniaeifard}
\address{
	F. Aliniaeifard: Research Center for Mathematics and Interdisciplinary Sciences,
	Shandong University \& Frontiers Science Center for Nonlinear Expectations, Ministry of Education,
	Qingdao, Shandong 266237, P.R.~China}
\email{farid@sdu.edu.cn}
\author{Nantel Bergeron}
\address{N. Bergeron: Mathematics and Statistics, York University, Ontario, Canada and Shandong Normal University, Jinan, P.R.~China}
\email{bergeron@yorku.ca}
\thanks{N. Bergeron: partially supported by an NSERC discovery grant}
\author{Huilan Li}
\address{H. Li: School of Mathematics and Statistics, Shandong Normal University, Jinan, P.R.~China}
\email{lihl@sdnu.edu.cn}
\author{Shu Xiao Li}
\address{
	S. Li: Research Center for Mathematics and Interdisciplinary Sciences,
	Shandong University \& Frontiers Science Center for Nonlinear Expectations, Ministry of Education,
	Qingdao, Shandong 266237, P.R.~China}
\email{lishuxiao@sdu.edu.cn}
\thanks{S. Li: partially supported by the Provincial Natural Science Foundation of Shandong, Project No. ZR2024QA026 and the Fundamental Research Funds for the Central Universities.}
\keywords{Symmetric functions in noncommuting variables, quasisymmetric functions,
coinvariants, nilpotence}

\begin{document}
\begin{abstract}
We study quotients of free associative algebras by the two-sided
ideals generated by positive-degree symmetric, quasisymmetric, or free
quasisymmetric polynomials.
Free quasisymmetric functions give a word realization of the
Malvenuto--Reutenauer Hopf algebra of permutations.
For quotients on $n$ variables, finiteness in the symmetric and
quasisymmetric cases follows from a theorem of Bergman--Isaacs
when $n!$ is invertible in the field.

Our main result is a nilpotence criterion
for families of relations compatible with restriction to subalphabets,
based on a filtration by word support. The criterion proves finiteness over
every field for quasisymmetric, $r$-quasisymmetric, and free
quasisymmetric coinvariants, with explicit degree bounds exponential
in $n$ for fixed $r$. In the quasisymmetric and $r$-quasisymmetric
cases, these can dramatically improve the bounds inherited from
Bergman--Isaacs when $n!$ is invertible, already for small alphabets.
A principal application is to the free quasisymmetric case, where,
for $n\geq3$, the group-invariant argument used for the symmetric
case does not apply.

For the symmetric quotient $C_n$, we prove that invertibility of
$n!$ is also necessary for finiteness.
In characteristic zero, an auxiliary presentation and the
Golod--Shafarevich inequality give
$\log\dim C_n\geq(1-o(1))n^5(\log n)^2$ as $n\to\infty$.
The same inequality shows that quasisymmetric and free quasisymmetric
coinvariant dimensions grow faster than every fixed power of $n!$,
over every field.
{For all three families on several freely noncommuting alphabets,
an exact tensor decomposition determines all multigraded Hilbert series
and, for symmetric coinvariants in characteristic zero, Frobenius
characteristics from one-alphabet data. In characteristic zero,
projection gives stable Schur multiplicities for compatible quotients
in mixed commuting, anticommuting, and free alphabets. The first
vanishing degree of the one-alphabet quotient controls both their
degree bounds and a uniform sufficient stable range.}
\end{abstract}
\maketitle

\section{Introduction}

Coinvariant algebras connect invariant theory, representation theory,
and geometry through explicit polynomial relations. In characteristic zero, the classical
coinvariant ring of the symmetric group is a graded model of its
regular representation and a presentation of the cohomology ring of
the complete flag variety~\cite{Borel53,Chevalley55}. Replacing its
symmetric relations by quasisymmetric polynomials gives a quotient
of Catalan dimension~\cite{AB03,ABB04}, now realized as the cohomology
ring of a toric complex in the flag variety~\cite{BGNST26}.

We study these quotient constructions in a free associative
algebra, where the relations must control both the multiplicities
and the order of letters. The first question is whether these quotients are finite dimensional. 
We answer it for symmetric, quasisymmetric, and free quasisymmetric relations.
Our main contributions are a general nilpotence criterion based on
restriction to subalphabets and quantitative lower bounds for the
growth of these coinvariant algebras.
{We also obtain an exact factorization for multiple free alphabets
and a uniform stability argument for their mixed-alphabet quotients.}
More precisely, let $\kk$ be a field and let $R_n=\kk\inv{x_1,\ldots,x_n}$
be graded by word length. Set
\begin{equation}\label{eq:coinvariants}
\begin{alignedat}{3}
&C_n=R_n/I_n,  &&T_n=R_n/J_n,  &&{\SC_n=R_n/\SI_n,}\\
&I_n=\inv{\NCSym_n^+},\qquad &&J_n=\inv{\NCQSym_n^+}, \qquad&&\SI_n=\inv{\FQSym_n^+}.
 \end{alignedat}
\end{equation}
All ideals in this paper are homogeneous and two-sided, and a superscript $+$ denotes
the positive-degree part. Here $\NCSym_n=R_n^{S_n}$
for the action of the symmetric group $S_n$ permuting the variables, and $\NCQSym_n$ is the
specialization to $n$ letters of the algebra of quasisymmetric
functions in noncommuting variables, also {denoted}
$\operatorname{WQSym}$, see~\cite{BZ09,H99}.
In degree $d$, the monomial bases of $\NCSym_n$ and $\NCQSym_n$
are indexed by set partitions
and ordered set partitions of $[d]$, respectively, with at most $n$
blocks. The inclusion
$\NCSym_n\subseteq\NCQSym_n$ gives a surjection
$C_n\twoheadrightarrow T_n$.
The algebra of free quasisymmetric functions, $\FQSym$,
is a realization of the
Malvenuto--Reutenauer Hopf algebra of permutations~\cite{MR95,DHT02}, and we denote by $\FQSym_n$
its specialization to $n$ letters. Both $\NCSym_n$ and $\FQSym_n$
lie in $\NCQSym_n$, but neither contains the other for $n\geq2$.

The algebras $C_n$ and $T_n$  and
their degreewise dual inverse systems were studied in~\cite{ABL10}\footnote{The
arXiv version stated quasisymmetric finiteness as Theorem~5.1, based
on a proof announced by Aval and Bergeron. That proof did not succeed
and was withdrawn; the published version therefore states the result
as Conjecture~5.1.}.
The two-sided ideal convention retains multiplication in the quotient. It is important to note that
the left-ideal coinvariants of~\cite{BRRZ08} are different objects and in that case
are infinite dimensional for $n\geq2$.

We first record the finiteness statement for $C_n$. This case
follows easily from Bergman--Isaacs~\cite{BI73}, and we add a modular obstruction.

\begin{mainproposition}\label{prop:A}
	For every $n\geq1$, the symmetric coinvariant algebra $C_n$ is finite
	dimensional if and only if $\operatorname{char}(\kk)\nmid n!$. 
\end{mainproposition}

When
$\operatorname{char}(\kk)\nmid n!$, restriction to the alternating group improves the resulting
degree estimate for $n\geq3$ (Proposition~\ref{prop:alternating}).
The characteristic criterion extends to every faithful finite-dimensional
representation of a finite group $G$: its tensor-algebra coinvariants
are finite dimensional precisely when $|G|$ is invertible in $\kk$
(Remark~\ref{rem:general-coinvariants}).

In characteristic zero, abelianization maps these algebras
onto their commutative counterparts, but the noncommutative quotients
can be much larger. Already $C_4$ has dimension $946$, compared with
$24$ in commuting variables.  It is then natural to ask how fast these quotients grow. For this, we use character theory and the
Golod--Shafarevich inequality~\cite[Theorem 2.1]{Ershov12} and obtain the following theorem.
For a connected graded algebra $A$ generated in degree one, write
$\nu(A)=\min\{d\geq1:A_d=0\}$ for its first vanishing degree,
with $\nu(A)=\infty$ if no such degree exists.
\begin{maintheorem}\label{Thm:GS}
Suppose $\operatorname{char}(\kk)=0$. As $n\to\infty$,
\[
 \nu(C_n)\geq(1-o(1))n^5\log n,\qquad
 \log\dim C_n\geq(1-o(1))n^5(\log n)^2.
\]
\end{maintheorem}

For quasisymmetric coinvariants, the surjection from $C_n$
already gives finiteness when $n!$ is invertible.
Theorem~\ref{thm:B} gives a much stronger conclusion: over every
field, every monomial of degree at least $2^{n-1}$ vanishes in $T_n$.
In particular, this proves the finiteness conjecture
of~\cite{ABL10} over every field.
The new degree bound is exponential in $n$. Already at $n=5$,
our bound is $16$, compared with approximately
$6.16\times10^{3005}$ from the direct Bergman--Isaacs argument
when $5!$ is invertible.
Our criterion does not require
any group action.

\begin{maintheorem}\label{thm:B}
Let $X$ be an alphabet of size $n$, and let $L_S$ be homogeneous
two-sided ideals of $R_S=\kk\inv{x:x\in S}$ for all nonempty
$S\subseteq X$. Suppose every relation on $S$ lifts to a relation
on $Y$ whenever $S\subseteq Y\subseteq X$; explicitly,
\[
 L_S\subseteq\pi_{Y,S}(L_Y),
\]
where $\pi_{Y,S}$ sets letters outside $S$ equal to zero.
Let $K_S$ be the two-sided ideal of $R_S$ spanned by words on $S$
in which every letter of $S$ occurs.
If $K_S^{e_s}\subseteq L_S$ whenever $|S|=s$, for positive integers
$e_1,\ldots,e_n$, then
\[
 (R_X/L_X)_+^{\,b_n}=0,
 \qquad b_n=e_n\prod_{s=1}^{n-1}(1+e_s).
\]
This holds over every field.
\end{maintheorem}

For quasisymmetric relations, every full-support word is
invariant, so taking $e_s=1$ gives the bound $2^{n-1}$ for $T_n$ (Corollary~\ref{thm:TL-bound}).
The criterion also treats the quotients $T_n^{(r)}$ defined by the
$r$-quasisymmetric functions of~\cite{ALW24}, which interpolate
between quasisymmetric and symmetric functions in noncommuting
variables. A product of $r$ full-support words is invariant, so
the criterion applies over every field. Combining it with the
two- and three-variable estimates gives
\[
 \nu(T_n^{(r)})\leq4r(r+1)^{n-3},
\]
for $n\geq3,\ r\geq1$ (Corollary~\ref{cor:r-char-zero}). For fixed $r$,
this is again exponential in $n$. Already for $r=2$ and $n=5$,
the bound is $72$, compared with the refined inherited bound of
approximately $6.16\times10^{3005}$ (or $1.01\times10^{1460}$ from Proposition~\ref{prop:alternating})
when $5!$ is invertible.
For $T_n$, two-sided lifting and
exact local calculations give further improvements
(Proposition~\ref{prop:seven-bound} and Remark~\ref{rem:ncqsym-improvement}).

For free quasisymmetric relations on $n\geq3$ variables, the ideal
defining $\SC_n$ is not $S_n$-stable and
does not contain the ideal generated by positive-degree invariants
of any subgroup of $S_n$. Thus the Bergman--Isaacs argument used
for $C_n$ does not apply here, even by passage to a quotient of
coinvariants for a subgroup of the natural permutation group.
A product of two full-support words
is, however, a free quasisymmetric relation, so Theorem~\ref{thm:B} with { $e_1=e_2=1$ and $e_s=2$, for $s\geq 3$,}
gives finiteness over every field, with
\[\nu(\SC_n)\leq8\cdot3^{n-3},\] 
for $n\geq3$
(Corollary~\ref{thm:fqsym-bound}).

A calculation of lower bounds for $A_n=T_n$ or
$\SC_n$, using an argument similar to the proof of Theorem~\ref{Thm:GS},
also reveals substantial growth in these finite-dimensional quotients.
Over every field, as $n\to\infty$, we show
\[
 \nu(A_n)\geq(1-o(1))n\log n,\qquad
 \log\dim A_n\geq(1-o(1))n(\log n)^2
\]
(Theorem~\ref{thm:ordered-growth}). Thus no linear degree bound is
possible, and both dimensions eventually exceed every fixed power
of $n!$.

We next consider how these algebras change when we add alphabets.
In commuting
variables, Haiman's work connects diagonal coinvariants to Hilbert
schemes of points in the plane and determines their dimension and
bigraded character~\cite{Haiman02}. For freely noncommuting
variables, the ordered sequence of alphabet labels in a word separates
from its variable indices. This separation survives passage to
coinvariants and gives an exact tensor decomposition.

{Let $C_{n,m}$, $T_{n,m}$, and $\SC_{n,m}$ denote the quotients
defined by the corresponding colored relations on $m$ freely
noncommuting alphabets of size $n$ (Section~\ref{sec:diagonal}).}
Write $\Hilb(D;t)$ for the Hilbert
series of a graded space $D$, and in characteristic zero put
\[
 \mathcal F_n(\mathbf z;t)
 =\sum_{d\geq0}\operatorname{Frob}_{S_n}((C_n)_d;\mathbf z)t^d,
\]
where $\operatorname{Frob}_{S_n}$ is the Frobenius characteristic
and $\mathbf z$ denotes its symmetric-function variables.

\begin{maintheorem}\label{thm:C}
{Let $D=C,T$, or $\SC$. Over every field $\kk$, for all
$n,m\geq1$ and $d\geq0$,
\[
 (D_{n,m})_d\cong(D_n)_d\otimes(\kk^m)^{\otimes d}.
\]
The isomorphism is $\mathrm{GL}_m$-equivariant for linear substitutions
among alphabets. For $D=C$, it is also $S_n$-equivariant, with trivial
$S_n$-action on $\kk^m$.}
The direct sums over $d$ are compatible with multiplication.
{For all three families,}
\[
 \Hilb(D_{n,m};q_1,\ldots,q_m)
 =\Hilb(D_n;q_1+\cdots+q_m).
\]
In characteristic zero,
\[
 \operatorname{Frob}_{S_n}(C_{n,m};\mathbf z;q_1,\ldots,q_m)
 =\mathcal F_n(\mathbf z;q_1+\cdots+q_m).
\]
Here $q_a$ records degree in the $a$th alphabet.
\end{maintheorem}

Whenever $D_n$ is finite dimensional, its highest nonzero degree is
independent of $m$ and $\dim D_{n,m}=\Hilb(D_n;m)$.
Thus the dimension and character problems for arbitrarily many
alphabets reduce to the one-alphabet quotients.

{In characteristic zero, these first vanishing degrees also
control representation stability under addition of alphabets.
Corollary~\ref{cor:mixed-stability} gives Schur multiplicities independent
of the numbers of commuting, anticommuting, and free alphabets for
quotients obtained by projections compatible with alphabet substitutions.
The total degree is at most $\nu(D_n)-1$, and every Schur constituent
is visible once each alphabet space present has dimension at least
$\nu(D_n)-1$. In the symmetric case, this also includes the
$S_n$-character. Earlier results establish stability for commuting and
bosonic--fermionic symmetric coinvariants~\cite{Ber13,Ber20BF,Lentfer25}.
Our argument treats $C,T$, and $\SC$ uniformly, including mixtures
containing free alphabets.}

Section~\ref{sec:finite} fixes conventions and proves
Proposition~\ref{prop:A} and the alternating-group bound. Section~\ref{sec:bounds}
proves Theorem~\ref{thm:B} and applies the restriction criterion to quasisymmetric
and free quasisymmetric coinvariants. Section~\ref{sec:r-ncqsym}
treats $r$-quasisymmetric coinvariants, and Section~\ref{sec:sharp}
contains the small-rank computations and proves the uniform growth estimates
for all three families, including Theorem~\ref{Thm:GS}.
Section~\ref{sec:diagonal} proves
Theorem~\ref{thm:C} {and derives mixed-alphabet stability},
and Section~\ref{sec:remaining} formulates
further questions.

\medskip
\noindent
{\bf Acknowledgement.} We thank F. Bergeron for conversations on the multialphabet cases. The mathematical ideas were developed primarily by the authors, with assistance from OpenAI Codex, which also prepared an initial draft. The authors critically reviewed and revised the manuscript and take responsibility for its contents.

\section{Finite dimensionality}\label{sec:finite}

\subsection{Words and invariants}
Write $[n]=\{1,\ldots,n\}$ and give every variable degree $1$.
Let \(X\) be an alphabet. Denote by \(X^*\) the \emph{free monoid generated by} $X$, whose elements are all finite words over $X$, including the \emph{empty word} \(1\). For a word \(w=x_{i_1}x_{i_2}\cdots x_{i_d}\in X_n^*\), where $X_n=\{x_1,x_2,\ldots,x_n\}$, its \emph{support} \(\operatorname{supp}(w)\) is defined by \(\operatorname{supp}(w)=\{x_{i_1},x_{i_2},\ldots,x_{i_d}\},\) the set of letters appearing in $w$.
For a graded algebra $B$, write $B_+=\bigoplus_{d\geq1}B_d$ and
$\Hilb(B;t)=\sum_{d\geq0}\dim_\kk(B_d)t^d$ when its homogeneous
components are finite dimensional. If $B$ is connected and generated
in degree one, define
\[
 \nu(B)=\min\{d\geq1:B_d=0\},
\]
with $\nu(B)=\infty$ if no such degree exists. Then
$B_+^d=\bigoplus_{j\geq d}B_j$, so a finite $\nu(B)$ is the
nilpotence index of $B_+$ and the highest nonzero degree is
$\nu(B)-1$.

The symmetric group {$S_n$} acts on $R_n$ by permuting the variables:
\[
 \sigma\cdot(x_{i_1}\cdots x_{i_d})
   =x_{\sigma(i_1)}\cdots x_{\sigma(i_d)},
   \]
for $\sigma\in S_n$.
This is an action by graded algebra automorphisms, and
$\NCSym_n=R_n^{S_n}$. Its monomial basis of degree $d$ is indexed by set partitions
of $[d]$ with at most $n$ blocks: the basis element $m_\Pi$ is the sum
of words whose equal-letter positions are precisely the blocks of the set partition $\Pi$.
We use the conventions of~\cite{RS06,BRRZ08}.

For an ordered set partition $\mathcal A=(A_1,\ldots,A_r)$ of $[d]$,
define
\begin{equation}\label{eq:MA}
 M_{\mathcal A}(X_n)=
 \sum_{\substack{1\leq j_1<\cdots<j_r\leq n\\ i_p=j_q\ \text{if }p\in A_q}}x_{i_1}\cdots x_{i_d}.
 \qquad 
\end{equation}
The sum is zero if $r>n$. These elements span $\NCQSym_n$; see
\cite{BZ09,ABL10}. For example,
\[
 M_{(\{1,3\},\{2\})}(X_3)
   =x_1x_2x_1+x_1x_3x_1+x_2x_3x_2.
\]
Summing $M_{\mathcal A}$ over all orders of the blocks of the set partition $\Pi=\{A_1,\ldots,A_r\}$ gives
$m_\Pi$. Consequently,
\begin{equation}\label{eq:surjection}
 \NCSym_n\subseteq\NCQSym_n,\qquad I_n\subseteq J_n,
 \qquad C_n\twoheadrightarrow T_n.
\end{equation}

\begin{remark}\label{rem:Hivert}
On words, Hivert's  quasi-symmetrizing action~\cite{H99} of a simple transposition
$s_i$ exchanges $x_i$ and $x_{i+1}$ when at most one of these letters
occurs, and fixes a word when both occur. Its invariant sums are
\eqref{eq:MA}. It is not multiplicative: it fixes $x_ix_{i+1}$,
whereas the product of the images of $x_i$ and $x_{i+1}$ is
$x_{i+1}x_i$. The terminology ``Temperley--Lieb coinvariants'' refers
to the quasi-symmetrizing framework of~\cite{Hivert00,ABB05}; it does
not assert that this action descends to an action by automorphisms
on the quotient.
\end{remark}

\subsection{The Bergman--Isaacs theorem}\label{ssec:thmA}
Set
\begin{equation}\label{eq:BI-function}
 f(m)=\prod_{j=1}^{m}\left(\binom{m}{j}+1\right).
\end{equation}
We use the following result of
\cite[Proposition 2.3]{BI73}. Algebras in this statement need not have
an identity.

\begin{proposition}[Bergman--Isaacs]\label{prop:BI}
Let a finite group $G$ act by automorphisms on a $\kk$-algebra $B$ without $|G|$-torsion.
If $B^G=0$, then
$B^{f(|G|)}=0$, with $f$ as in~\eqref{eq:BI-function}.
If $G$ is solvable, the stronger conclusion $B^{|G|}=0$ holds.
\end{proposition}

The solvable assertion is~\cite[Proposition 1.3]{BI73}. The function
$f$ occurs in their Lemma~2.1. Later formulations and the problem of
improving this bound are discussed in~\cite[Section 2]{BG21}. We apply it to the symmetric coinvariant algebra.

\begin{corollary}\label{thm:finite}
If $\operatorname{char}(\kk)\nmid n!$, $(C_n)_+^{f(n!)}=0$.
In particular, $C_n$ and $T_n$ are finite dimensional.
\end{corollary}
\begin{proof}
The ideal $I_n$ is homogeneous and $S_n$-stable, so $S_n$ acts by
automorphisms on $C_n$ and on its nonunital algebra $B=(C_n)_+$.
We claim that $B^{S_n}=0$. If $\overline{\mathbf r}\in B^{S_n}$, choose a
positive-degree representative $\mathbf r\in R_n$. Its Reynolds average
\[
 \mathcal R(\mathbf r)=\frac1{n!}\sum_{\sigma\in S_n}\sigma\cdot\mathbf r
\]
belongs to $\NCSym_n^+\subseteq I_n$, while its class is $\overline{\mathbf r}$.
Thus $\overline{\mathbf r}=0$. Proposition~\ref{prop:BI} gives $B^{f(n!)}=0$.

The algebra $C_n$ is generated in degree $1$, so $B^d$ is the sum of
its homogeneous components of degrees at least $d$. Hence every word
of degree at least $f(n!)$ lies in $I_n$, and the finitely many shorter
words span $C_n$. Finally, the surjection in~\eqref{eq:surjection}
proves that $T_n$ is finite dimensional as well.
\end{proof}

\begin{proof}[Proof of Proposition~\ref{prop:A}]
If $\operatorname{char}(\kk)\nmid n!$,
then $n!$ is invertible, and Corollary~\ref{thm:finite} applies.

Suppose now that $p=\operatorname{char}(\kk)$ divides $n!$,
so $0<p\leq n$.
Define a surjective algebra homomorphism $\varphi:R_n\to\kk[t]$ by
\[
 \varphi(x_i)=\begin{cases}
 {t,} & i\leq p,\\
 {0,} & i>p.
 \end{cases}
\]
For a set partition $\Pi=\{A_1,\ldots,A_r\}$ of $[d]$, with
$d>0$ and $r\leq n$, let $m_{\Pi}$ be the corresponding
monomial basis element of $\NCSym_n$. If $r>p$, then
$\varphi(m_{\Pi})=0$. If $1\leq r\leq p$, then
\[
 \varphi(m_{\Pi})=p(p-1)\cdots(p-r+1)t^d=0.
\]
Therefore, $\varphi(I_n)=0$, so $\varphi$ induces a surjection
$C_n\twoheadrightarrow\kk[t]$. In particular, $C_n$ is infinite
dimensional.
\end{proof}

Under the hypotheses of Corollary~\ref{thm:finite}, solvability of $S_n$
for $n\leq4$ improves the bound for $C_n$ to $n!$.
In characteristic zero, the exact first {vanishing} degrees
for $n\leq4$ are $1,2,4,10$; Section~\ref{sec:sharp} records the
Hilbert polynomials and a lower bound for $n=5$.

\begin{remark}[Finite-group coinvariants]\label{rem:general-coinvariants}
The argument is independent of the commutation rules for the variables.
Let $P$ be a finitely generated connected graded associative algebra
over a field $\kk$, and let a finite group $G$ act on $P$ by graded
algebra automorphisms. Assume that $|G|$ is invertible in $\kk$;
equivalently, $\operatorname{char}(\kk)\nmid|G|$. Then
\[
 D=P/\inv{(P^G)_+}
 \qquad\text{satisfies}\qquad D_+^{f(|G|)}=0,
\]
and is finite dimensional. Indeed, averaging gives $(D_+)^G=0$, so
Proposition~\ref{prop:BI} applies; products of fewer than $f(|G|)$
generators span $D$. This includes polynomial algebras in finitely many
commuting, anticommuting, or freely noncommuting variables, and mixtures
of these, whenever the action preserves the defining relations.

For general $P$, finite dimensionality does not require this
characteristic hypothesis. For example, in characteristic $2$, the
faithful action of $C_2$ interchanging the commuting variables $x,y$
has coinvariant algebra
\[
 \kk[x,y]/(x+y,xy)\cong\kk[t]/(t^2).
\]

For a finite-dimensional $G$-module $V$, we write
$D_G(V)=T(V)/\inv{T(V)^G_+}$ for the case of the tensor algebra $T(V)$.
If the action on $V$ is faithful, then $D_G(V)$ is finite dimensional
if and only if $|G|$ is invertible in $\kk$. To see the converse,
suppose that $\operatorname{char}(\kk)=p>0$ divides $|G|$ and choose
$g\in G$ of order $p$. Since $(g-1)^p=0$ and $g$ acts nontrivially,
a Jordan basis for $g$ contains a block $x_1,\ldots,x_r$, $r\geq2$,
with
\[
 gx_1=x_1,\qquad gx_i=x_i+x_{i-1},
\]
for $\quad2\leq i\leq r.$  Writing $[w]f$ for the coefficient of a word $w$ in $f$, every
$f\in T(V)^G_d$, $d\geq1$, satisfies
\[
 0=[x_{r-1}x_r^{d-1}](gf-f)=[x_r^d]f.
\]
Indeed, no other basis variable produces $x_r$ under $g$, and only
$x_{r-1}$ and $x_r$ produce $x_{r-1}$. Hence sending $x_r$ to $t$
and all other Jordan basis variables to zero annihilates every
positive-degree invariant and induces a surjection
$D_G(V)\twoheadrightarrow\kk[t]$. Thus $D_G(V)$ is infinite
dimensional. For an arbitrary action, the same criterion holds with
$G$ replaced by its image in $\mathrm{GL}(V)$.
\end{remark}

\subsection{The alternating-group bound}\label{ssec:Cn}
Let $\mathfrak A_n$ denote the alternating subgroup of $S_n$.
For a finite group $H$ acting on
an algebra $B$ and $b\in B$, put $\tr_H(b)=\sum_{h\in H}h\cdot b$.
We use a direct specialization of~\cite[Lemma 2.2(c)]{BI73}.

\begin{lemma}\label{lem:nilpotent-trace}
Let a finite group $H$ act by automorphisms on a $\kk$-algebra $B$,
not necessarily with an identity, and assume $|H|$ is invertible in $\kk$.
If $(B^H)^d=0$ for an integer $d\geq1$, then
\[
 B^{(f(|H|)+1)^d}=0.
\]
\end{lemma}
\begin{proof}
Write $B^\sharp$ for $B$ with an identity adjoined. Since
$\tr_H(B)=B^H$, Lemma~2.2(c) of~\cite{BI73}, with its left ideal
equal to $B$, gives
\[
 B^{(f(|H|)+1)^d}\subseteq B\bigl(\tr_H(B)\bigr)^dB^\sharp=0.
\]
The factors $|H|$ in the integral statement can be cancelled
because $|H|$ is invertible in $\kk$.
\end{proof}

\begin{proposition}\label{prop:alternating}
Let $n\geq2$ and assume $\operatorname{char}(\kk)\nmid n!$.
Put $B=(C_n)_+$. Then $(B^{\mathfrak A_n})^2=0$ and
\[
 B^{(f(n!/2)+1)^2}=0.
\]
\end{proposition}
\begin{proof}
We have $B^{S_n}=0$. Since $\mathfrak A_n$ is normal of index $2$, the quotient
$S_n/\mathfrak A_n$ acts on $B^{\mathfrak A_n}$. For an odd permutation $\tau$ and
$b\in B^{\mathfrak A_n}$, the element $b+\tau b$ is fixed by $S_n$, hence
zero. Thus $\tau b=-b$. If $b,c\in B^{\mathfrak A_n}$, then $bc\in B^{\mathfrak A_n}$
and
\[
 \tau(bc)=(\tau b)(\tau c)=(-b)(-c)=bc.
\]
Consequently $bc$ is fixed by $S_n$, so $bc=0$.
Apply Lemma~\ref{lem:nilpotent-trace} with $H=\mathfrak A_n$ and $d=2$.
\end{proof}

For $n\geq3$, this improves on the bound $f(n!)$. Indeed, for
$m\geq2$, Vandermonde's identity gives
$\binom{2m}{j}\geq2\binom mj$ for $1\leq j<m$ and
$\binom{2m}{m}\geq m^2+2$. {Therefore,}
\[
 \frac{f(2m)}{f(m)^2}
 =\frac{\binom{2m}{m}+1}{2}
  \prod_{j=1}^{m-1}
  \left(\frac{\binom{2m}{j}+1}{\binom mj+1}\right)^2
 >\frac72>
 \left(1+\frac1{f(m)}\right)^2,
\]
{as} $f(m)\geq6$. Both bounds remain large: for $n=5$ the
alternating-group estimate is approximately $10^{1460}$, compared
with $10^{3005}$ for $f(5!)$.

\section{Restriction and degree bounds}\label{sec:bounds}

\subsection{A nilpotence criterion from word supports}\label{ssec:Tn}
Throughout this subsection, $\kk$ is an arbitrary field. For a finite
alphabet $X$ and $\varnothing\ne S\subseteq X$, let
$R_S=\kk\inv{x:x\in S}$. 
The restriction homomorphism $\pi_{X,S}:R_X\to R_S$
fixes each letter in $S$ and sends every letter outside $S$ to zero.
For $1\leq s\leq|X|$,
put
\[
 K_{X,s}=\Span_\kk\{w\in X^*:|\supp(w)|\geq s\}.
\]
These are two-sided monomial ideals; in particular, $K_{X,|X|}$
is spanned by the full-support words. For $X=X_n$, we abbreviate
$K_{X,s}$ to $K_{n,s}$.

\begin{definition}\label{def:compatible}
A family of homogeneous two-sided ideals $(L_S)_{\varnothing\ne S\subseteq X}$, 
where $L_S$ is an ideal of $R_S$, is said to be \emph{compatible with restriction} if
\[
L_S\subseteq\pi_{Y,S}(L_Y),
\qquad \varnothing\ne S\subseteq Y\subseteq X.
\]
\end{definition}
Thus every relation on a smaller alphabet admits a lift to a larger
one. Equality is not required.

Let $F_s$ be the image of $K_{X,s}$ in $R_X/L_X$. These ideals form
the support filtration
\begin{equation}\label{eq:filtration}
 (R_X/L_X)_+=F_1\supseteq F_2\supseteq\cdots\supseteq F_{|X|}.
\end{equation}
The following lifting step is the key to Theorem~\ref{thm:B}.

\begin{lemma}\label{lem:filtration}
Suppose $(L_S)$ is compatible with restriction, and let
$1\leq s<|X|$. If every word of degree $b_s$ whose support is a subset of $S$ belongs to $L_S$
for each $S\subseteq X$ of size $s$, then
\[
 F_sF_1^{b_s}\subseteq F_{s+1}.
\]
\end{lemma}
\begin{proof}
Consider a product of two words $uv$, where $u$ uses at least $s$ letters and
$v$ has degree at least $b_s$. If $uv$ uses more than $s$ letters,
its class lies in $F_{s+1}$. Otherwise
$S=\supp(u)=\supp(uv)$ has size $s$, and $v\in L_S$.
Choose a lift $\widetilde v\in L_X$ with
$\pi_{X,S}(\widetilde v)=v$. Every word of $\widetilde v-v$ has a
letter outside $S$. Since $u$ already uses every letter of $S$,
\[
 uv=u\widetilde v-u(\widetilde v-v)
 \in L_X+K_{X,s+1}.
\]
Such products span the required ideal product.
\end{proof}

\begin{theorem}\label{thm:general}
Let $|X|=n$ and suppose $(L_S)$ is compatible with restriction.
For positive integers $e_1,\ldots,e_n$, assume
\[
 K_{S, s}^{e_s}\subseteq L_S,\qquad|S|=s.
\]
Then
\[
 (R_X/L_X)_+^{\,b_n}=0,
 \qquad b_n=e_n\prod_{s=1}^{n-1}(1+e_s).
\]
\end{theorem}
\begin{proof}
For a singleton alphabet, the hypothesis gives $b_1=e_1$.
Inductively suppose that $b_s$ is a vanishing degree for every
alphabet of size $s<n$. Applying Lemma~\ref{lem:filtration}
successively gives
\[
 F_1^{\,1+b_1+\cdots+b_{s-1}}\subseteq F_s,
 \qquad1\leq s\leq n.
\]
Since $F_n^{e_n}=0$, we may take
\begin{equation}\label{eq:support-recurrence}
 b_n=e_n\left(1+\sum_{s=1}^{n-1}b_s\right).
\end{equation}
Starting from $b_1=e_1$, this recurrence gives
$1+\sum_{s<n}b_s=\prod_{s<n}(1+e_s)$, proving the formula.
\end{proof}

The full-support ideal $K_S$ in Theorem~\ref{thm:B} is
$K_{S,|S|}$ in the present notation. Thus,  Theorem~\ref{thm:general}
 proves Theorem~\ref{thm:B}. 
 The support argument of our proof uses deletion of variables in place of averaging. Its
increasing subsets resemble those in~\cite[Lemma 2.1]{BI73}.

Our proof also shows that any
improved vanishing degrees on smaller alphabets can be inserted
in~\eqref{eq:support-recurrence}. 

\begin{proposition}[A sharper support bound]\label{prop:seven-bound}
Under the hypotheses of Theorem~\ref{thm:general},\hfill\break
if $e_1=e_2=e_3=1$, then
\begin{equation}\label{eq:refined-general-bound}
 (R_X/L_X)_+^{\,7e_n\prod_{s=4}^{n-1}(1+e_s)}=0,
 \qquad n\geq4.
\end{equation}
\end{proposition}
\begin{proof}
The essential improvement is the uniform inclusion $A_+^7\subseteq F_4$,
where $A=R_X/L_X$. It requires only $e_1=e_2=e_3=1$.

Replace $L_S$ by $\pi_{X,S}(L_X)$, which enlarges the smaller ideals
and makes restriction exact without changing $A$. After extending to
an infinite field, if necessary, the singleton restrictions allow a
linear relation with every coefficient nonzero. Rescaling the letters
makes it $\sum_{x\in X}x$. It restricts to the corresponding sum on
each $S$. Every two-letter quotient consequently vanishes in degree
$2$: the mixed products vanish, and the linear relation kills the
squares.

For a triple $S=\{a,b,c\}$, eliminate $c=-a-b$. The six full-support
cubics give
\[
 a^2b=aba=ba^2=\eta,\qquad ab^2=bab=b^2a=-\eta.
\]
Exact restriction to the pairs allows a quadratic relation whose
expression after elimination is
\[
 q=\alpha a^2+\beta ab+\gamma ba+\delta b^2,
 \qquad \alpha\delta(\alpha-\beta-\gamma+\delta)\ne0.
\]
Indeed, exact restriction and the vanishing of the pair quotients in
degree two give $\pi_{S,P}((L_S)_2)=(R_P)_2$ for each pair $P\subset S$.
Setting the third letter to zero and evaluating the two letters of $P$
at $1,-1$ therefore gives a nonzero linear functional on $(L_S)_2$.
Choose $q$ outside the kernels of these three functionals.
Multiplying $q$ on the left by $a$ and $b$ gives
\[
 \alpha a^3+(\beta+\gamma-\delta)\eta=0,
 \qquad
 \delta b^3+(\alpha-\beta-\gamma)\eta=0.
\]
Thus the cubic quotient is spanned by $\eta$, and
\[
 bcb=\frac{\alpha-\beta-\gamma+\delta}{\delta}\,\eta.
\]
Thus any specified $bcb$, with $b\ne c$ in $S$, spans that quotient.

We use the following two-sided form of the lifting argument: if
$S\subsetneq X$, $h\in L_S$, and the concatenation of words $p,z$
has support $S$, then
\begin{equation}\label{eq:context-lift}
 phz\in F_{|S|+1}\quad\text{in }A.
\end{equation}
Indeed, for a lift $\widetilde h\in L_X$, we have
$phz=-p(\widetilde h-h)z$ in $A$. The context $p,z$ supplies every
letter of $S$ to each error term, which also contains a letter outside
$S$. Now $F_1^2\subseteq F_2$ and $F_2^2\subseteq F_3$
give $A_+^4\subseteq F_3$. Modulo $F_4$, a degree-seven class is
therefore spanned by $uv$, where $|u|=4$, $|v|=3$, and both use only
$S=\supp(u)$, a triple. Let $a$ be the repeated letter of $u$, and
$b,c$ the other letters. The cubic calculation and lifting replace
$v$ by a multiple of $bcb$.

In $ubcb$, start an interval at the second $a$. If the original $c$
follows this $a$, end at the first appended $b$; otherwise end at the
appended $c$. The interval contains $a,b,c$, and its complementary
prefix and suffix together contain $a,b,c$. The word in the interval
belongs to $L_S$, so
\eqref{eq:context-lift} puts $ubcb$ in $F_4$. Since $A$ is generated
in degree one and $F_4$ is an ideal, this proves $A_+^7\subseteq F_4$.

Lifting gives
$F_s^{1+e_s}\subseteq F_{s+1}$ for $s<n$: one full-support factor supplies the
context for a relation consisting of $e_s$ further factors.
Together with $F_n^{e_n}=0$, this proves~\eqref{eq:refined-general-bound}.
{The} conclusion descends from the field extension.
\end{proof}

The lifting hypothesis is essential. Full-support vanishing by
itself leaves an infinite-dimensional quotient: for $n\geq2$,
the substitution $x_1\mapsto t$, $x_i\mapsto0$ for $i>1$ gives
$R_n/K_{n,n}\twoheadrightarrow\kk[t]$. Even adjoining
$x_1+\cdots+x_n$ does not suffice when $n\geq3$, as shown by
$x_1\mapsto t$, $x_2\mapsto-t$, and $x_i\mapsto0$ for $i>2$.

\subsection{Quasisymmetric coinvariants}\label{ssec:quasisymmetric}
For an ordered subalphabet $S$, let $J_S$ be its quasisymmetric
ideal, with the inherited order. Formula~\eqref{eq:MA} gives
\begin{equation}\label{eq:restriction}
 \pi_{Y,S}(M_{\mathcal A}(Y))=M_{\mathcal A}(S),
 \qquad \pi_{Y,S}(J_Y)=J_S.
\end{equation}
Moreover, a full-support word on $S$ is itself an element
$M_{\mathcal A}(S)$ with $|S|$ blocks.

\begin{corollary}\label{thm:TL-bound}
Over every field, for all $n\geq1$,
\[
 (T_n)_+^{2^{n-1}}=0.
\]
In particular, $T_n$ is finite dimensional and
$\deg\Hilb(T_n;t)\leq2^{n-1}-1$.
\end{corollary}
\begin{proof}
The family $(J_S)$ satisfies both hypotheses of
Theorem~\ref{thm:general} with $e_s=1$.
\end{proof}

\begin{remark}[Sharper quasisymmetric bounds]\label{rem:ncqsym-improvement}
Proposition~\ref{prop:seven-bound} gives $\nu(T_n)\leq7\cdot2^{n-4}$ for
$n\geq4$ over every field. In characteristic zero, the explicit
quasisymmetric relations give the stronger bound
\begin{equation}\label{eq:ncqsym-improved-bound}
 \nu(T_n)\leq23\cdot2^{n-7},\qquad n\geq7.
\end{equation}
Here the additional local steps are computer-assisted. We describe
them to distinguish the finite calculations from the argument valid
for every $n$.

For a four-letter alphabet $S$, put
\[
 H_{S,8}=\Span_\kk\{phz:h\in(J_S)_k,\ k\geq1,
       \ |p|+k+|z|=8,\ \supp(pz)=S\},
\]
where $p,z$ are words on $S$. Exact linear algebra gives
$H_{S,8}=(K_{S,4})_8$. Of the $40\,824$ full-support words of degree
$8$, precisely $22\,536$ have a contiguous quartic whose complementary
context still has full support; these lie in $H_{S,8}$ because
$(T_4)_4=0$. On the remaining $18\,288$ words, the rows
$pM_{\mathcal A}(S)z$ with $1\leq|\mathcal A|\leq3$ and full-support
context have rank $18\,288$ modulo $2$, hence also over $\mathbb Q$.
Here $|\mathcal A|$ is the total size of the ordered set partition.
Together with $A_+^8\subseteq F_4$, this calculation and
\eqref{eq:context-lift} prove
\[
 A_+^8\subseteq F_5,\qquad A=T_n,
\]
uniformly in the ambient alphabet, for $n\geq5$.

The small-rank calculation in Section~\ref{ssec:small-rank} gives
$(T_5)_6=0$. A graded relation-matrix calculation over $\mathbb F_2$
gives $(T_6)_9=0$, and therefore the same vanishing in characteristic
zero. Lemma~\ref{lem:filtration} now gives
\[
 A_+^{14}\subseteq F_6,\qquad A_+^{23}\subseteq F_7.
\]
Repeated squaring, followed by $F_n=0$, proves
\eqref{eq:ncqsym-improved-bound}.
\end{remark}

\subsection{Malvenuto--Reutenauer Hopf algebra and standardized coinvariants}\label{ssec:fqsym}
For a word $w=x_{i_1}x_{i_2}\cdots x_{i_m}$,
its standardization $\std(w){=\sigma_1\sigma_2\cdots\sigma_m}\in S_m$ is defined by
\[
 \sigma_t=\#\{j:i_j<i_t\}+\#\{j\leq t:i_j=i_t\}.
\]
For $\sigma\in S_m$, define
\begin{equation}\label{eq:FB}
F_\sigma(X_n)=\sum_{\substack{w\in X_n^*\\\std(w)^{-1}=\sigma}}w.
\end{equation}
The algebra of free quasisymmetric polynomials on $X_n$ is
\[
 \FQSym_n=\bigoplus_{m\geq0}\Span_\kk\{F_\sigma(X_n):\sigma\in S_m\}.
\]

On a countably infinite ordered alphabet, these series realize the
Malvenuto--Reutenauer Hopf algebra~\cite{MR95}; see
Duchamp, Hivert, and Thibon~\cite[Section 3.1]{DHT02} for the word realization
and its basis conventions. This Hopf algebra is self-dual: if
$G_\sigma=\sum_{\std(w)=\sigma}w$, then
$F_\sigma=G_{\sigma^{-1}}$ is the basis dual to $G_\sigma$.
The algebra $\FQSym_n$ is its finite-alphabet specialization.
The coproduct is realized by splitting the infinite alphabet, as
discussed in Section~\ref{ssec:hopf}.

The standard ideal of $R_n$ and its coinvariant algebra are
\[
 \SI_n=\inv{\FQSym_n^+}
      =\inv{F_\sigma(X_n):\sigma\in\bigcup_{m\geq1}S_m}
      \subseteq R_n,
 \qquad \text{and} \qquad  \SC_n=R_n/\SI_n,
\]
respectively.
For any subalphabet $S$, define $\SI_S$ using the inherited order.
Formula~\eqref{eq:FB} gives
\begin{equation}\label{eq:restrictionMR}
 \pi_{Y,S}(F_{\sigma}(Y))=F_{\sigma}(S),
 \qquad \pi_{Y,S}(\SI_Y)=\SI_S.
\end{equation}

As shown in~\cite{DHT02}, every
$F_\sigma(X_n)$ is a sum of monomial quasisymmetric polynomials,
and we have the natural inclusion
\[
 \FQSym_n\subseteq\NCQSym_n.
\]
Recall that we also have $ \NCSym_n\subseteq\NCQSym_n$, but
for $n\geq2$, neither
$\NCSym_n$ nor $\FQSym_n$ contains the other. Indeed,
$p_2=\sum_i x_i^2$ belongs to $\NCSym_n$ but not to $\FQSym_n$:
its coefficients on $x_1^2$ and $x_1x_2$ differ, although these words
have the same standardization. Conversely,
$F_{21}(X_n)=\sum_{i>j}x_ix_j$ is not invariant under permuting
the variables.

\begin{remark}\label{rem:fqsym-BI}
The inclusion $\FQSym_n\subseteq\NCQSym_n$ gives
$\SI_n\subseteq J_n$ and a surjection $\SC_n\twoheadrightarrow T_n$.
Finiteness of $T_n$ therefore does not imply finiteness of $\SC_n$.
For $n\geq3$, the Bergman--Isaacs argument used for $C_n$ does not
apply to $\SC_n$, even indirectly through an ideal generated by
invariants of a subgroup of the natural permutation group.

To see both obstructions, work first with $X_3=\{x_1,x_2,x_3\}$.
The linear relation $F_1=x_1+x_2+x_3$ eliminates $x_3=-x_1-x_2$.
Since $F_{12}+F_{21}=F_1^2$, the remaining degree-two relations
span the single line generated by
\[
 q=x_1^2+x_1x_2+x_2^2,
 \qquad F_{21}=x_2x_1+x_3x_1+x_3x_2=-q.
\]
Thus $x_1x_2-x_2x_1\notin\SI_3$. But interchanging $x_1$ and $x_2$ changes
$F_{21}$ by $x_1x_2-x_2x_1$, so $\SI_3$ is not $S_3$-stable.
Moreover, the symmetric polynomial $p_2=x_1^2+x_2^2+x_3^2$ satisfies
\[
 p_2=2x_1^2+x_1x_2+x_2x_1+2x_2^2\equiv x_2x_1-x_1x_2\ne0\pmod{\SI_3}.
\]
Exact restriction in~\eqref{eq:restrictionMR} gives the same
obstructions for every $n\geq3$: $\SI_n$ is not $S_n$-stable
and $p_2(X_n)\notin\SI_n$. Since $p_2(X_n)$ is fixed by every
subgroup $H\leq S_n$, we have
\[
 \inv{(R_n^H)^+}\not\subseteq\SI_n
 \qquad(H\leq S_n,\ n\geq3).
\]
Consequently, none of these invariant ideals provides a subideal
from which the Bergman--Isaacs theorem could imply finiteness of
$\SC_n$. The proof below instead uses Theorem~\ref{thm:B} alone.
\end{remark}

\begin{lemma}\label{lem:SI}
If $u$ and $v$ are words with {full support} on $X_n$, then
$F_{\std(uv)^{-1}}(X_n)=uv$.
\end{lemma}
\begin{proof}
Put $w=uv$, $m=|w|$, and $\sigma=\std(w)^{-1}$. The permutation
$\sigma$ lists the positions of $x_1$ in increasing order, then those
of $x_2$, and so on. At each boundary between consecutive letters,
the last occurrence of $x_i$ lies in $v$ and the first occurrence of
$x_{i+1}$ lies in $u$. Thus $\sigma$ has exactly $n-1$ descents.

If $z=z_1\cdots z_m$ is another word on $X_n$ with
$\std(z)^{-1}=\sigma$, then
\[
 z_{\sigma_1}\leq\cdots\leq z_{\sigma_m},\qquad
 z_{\sigma_j}<z_{\sigma_{j+1}} \text{whenever }\sigma_j>\sigma_{j+1}.
\]
The $n-1$ forced strict increases on an $n$-letter alphabet determine
this sequence uniquely: it starts at $x_1$, increases by one letter
at each descent of $\sigma$, and is constant elsewhere. Hence
$z=w$, proving the claim.
\end{proof}

\begin{corollary}\label{thm:fqsym-bound}
Over every field, $\SC_n$ is finite dimensional for every $n\geq1$.
More precisely, $\SC_1=\kk$, $\SC_2\cong\kk[t]/(t^2)$, and
\[
 (\SC_n)_+^{8\cdot3^{n-3}}=0,\qquad n\geq3.
\]
In particular, $\deg\Hilb(\SC_n;t)\leq8\cdot3^{n-3}-1$ for $n\geq3$.
\end{corollary}
\begin{proof}
On one letter, $F_1=x_1$ gives $\SC_1=\kk$. On two letters,
$F_1=x_1+x_2$ and $F_{21}=x_2x_1$ give $x_2=-x_1$ and $x_1^2=0$;
since the only linear relation is $F_1$, this gives
$\SC_2\cong\kk[t]/(t^2)$. Thus we may take $e_1=e_2=1$.
Lemma~\ref{lem:SI} gives $K_S^2\subseteq\SI_S$ for every subalphabet
$S$, with $K_S$ as in Theorem~\ref{thm:B}. Together with
\eqref{eq:restrictionMR}, Theorem~\ref{thm:B} with $e_s=2$ for
$s\geq3$ yields the stated bound.
\end{proof}

\subsection{Alphabet realizations and Hopf structure}\label{ssec:hopf}
The restriction property extends beyond quasisymmetric functions.
Suppose a connected graded algebra $H$ is realized by bounded-degree
series on a countable alphabet, with specializations $H_S\subseteq R_S$
such that $\pi_{Y,S}(H_Y)=H_S$. Then

\[
 \pi_{Y,S}(\inv{H_Y^+})=\inv{H_S^+}.
\]
Thus the restriction hypothesis is built into the finite-alphabet
realizations with this specialization property. To apply
Theorem~\ref{thm:general}, it remains to bound the nilpotence of the
full-support ideal. Section~\ref{sec:r-ncqsym} carries this out for
$r$-quasisymmetric functions.

Such realizations often carry a Hopf structure as well. Assume that
evaluation of each element is compatible with restriction,
order-preserving relabelling, and iterated ordered alphabet sums.
Take disjoint
alphabets $Y,Z$, ordered with every letter of $Y$ before every letter
of $Z$ when order matters. Evaluate the alphabet sum in a tensor
product: letters within each alphabet remain free, while letters of
$Y$ commute with letters of $Z$. If this evaluation
has a finite decomposition
\[
 h(Y+Z)=\sum_i h_i^{(1)}(Y)h_i^{(2)}(Z),\qquad
 h_i^{(1)},h_i^{(2)}\in H,
\]
define $\Delta h=\sum_i h_i^{(1)}\otimes h_i^{(2)}$.
Evaluation is multiplicative. Splitting into three alphabets gives coassociativity, evaluation on
the empty alphabet gives the counit, and connectedness supplies the
antipode. For example, in $\NCQSym$, the formula is
\cite[Section 5.2]{BZ09}
\[
 \Delta M_{(A_1,\ldots,A_r)}
 =\sum_{i=0}^r
 M_{\operatorname{std}(A_1,\ldots,A_i)}\otimes
 M_{\operatorname{std}(A_{i+1},\ldots,A_r)},
\]
where standardization relabels the positions increasingly, and the
empty ordered set partition represents $1$. The infinite-alphabet
algebras $\NCSym$, $\FQSym$, and $\NCQSym$ have this structure; their individual
finite-alphabet specializations need not inherit a coproduct.
Our finiteness proofs use only multiplication and restriction.

There is nevertheless a natural connection to the quotients. For
$A_S=R_S/\inv{H_S^+}$ and finite $Y,Z$, alphabet splitting induces a
surjective algebra map
\begin{equation}\label{eq:quotient-alphabet-splitting}
 A_{Y+Z}\twoheadrightarrow A_Y\otimes A_Z,
 \qquad y\longmapsto y\otimes1,\quad z\longmapsto1\otimes z.
\end{equation}
Indeed, every coproduct term of a positive-degree element has a
positive-degree factor in $H$, so the defining relations map to zero.
In particular, for $A=C,T,\SC$, these maps give
$\nu(A_{p+q})\geq\nu(A_p)+\nu(A_q)-1$ whenever the quotients are
finite dimensional. To obtain upper bounds, one would also need
information about their kernels.

\begin{question}\label{q:hopf}
Can the kernels of the alphabet-splitting maps
\eqref{eq:quotient-alphabet-splitting} be controlled through the support
filtration, yielding sharper degree bounds or recursive descriptions
of the coinvariant algebras?
\end{question}

\section{Noncommutative \texorpdfstring{$r$}{r}-quasisymmetric coinvariants}\label{sec:r-ncqsym}

We apply the support criterion to the $r$-quasisymmetric functions in
noncommuting variables introduced in~\cite[Sections 8.4 and 13]{ALW24}\footnote{
This also appears in an unpublished manuscript of {F. Hivert~\cite{H05}.}}.
Throughout this section, $r\geq1$ is an integer and $\kk$ is an arbitrary
field, unless otherwise stated.

\subsection{The action and restriction to subalphabets}
For a word $w$, let $a_i(w)$ be the number of occurrences of $x_i$.
The $r$-quasi-symmetrizing action of a simple transposition is
\[
 s_i\mathbin{\star_r}w=
 \begin{cases}
  s_i\cdot w,&a_i(w)<r\text{ or }a_{i+1}(w)<r,\\
  w,&a_i(w)\geq r\text{ and }a_{i+1}(w)\geq r,
 \end{cases}
\]
where $s_i\cdot w$ exchanges every $x_i$ and $x_{i+1}$. Let
$H_n^{(r)}=r\text{-}\NCQSym_n$ denote the invariant space.
The integral product formulas for its orbit-sum basis
 show that $H_n^{(r)}$ is a subalgebra
over every field. Define
\begin{equation}\label{eq:r-coinvariants}
 J_n^{(r)}=\inv{(H_n^{(r)})_+},\qquad
 T_n^{(r)}=R_n/J_n^{(r)}.
\end{equation}
The ideal is homogeneous and two-sided. For $r=1$, these are $J_n$
and $T_n$. Moreover,
\[
 \NCSym_n\subseteq H_n^{(r+1)}\subseteq H_n^{(r)}
 \subseteq\NCQSym_n,
\]
and the degreewise limit as $r\to\infty$ is $\NCSym_n$.
When $\operatorname{char}(\kk)\nmid n!$, the resulting surjection
$C_n\twoheadrightarrow T_n^{(r)}$ supplies the bounds of
Corollary~\ref{thm:finite} and Proposition~\ref{prop:alternating}.
The support criterion will give bounds exponential in $n$ for every
fixed $r$, over arbitrary fields.

For any ordered subalphabet $S$, define $H_S^{(r)}$ and $J_S^{(r)}$
using its inherited order.

\begin{lemma}\label{lem:r-restriction}
For every nonempty $S\subseteq Y$, specialization to $S$ satisfies
\[
 \pi_{Y,S}(H_Y^{(r)})=H_S^{(r)},\qquad
 \pi_{Y,S}(J_Y^{(r)})=J_S^{(r)}.
\]
\end{lemma}
\begin{proof}
An orbit of words is determined by the partition of their positions
into equal-letter blocks, together with the order, by variable index,
of the blocks of size at least $r$. Indeed, adjacent labels may be
exchanged unless both correspond to such blocks. The smaller blocks
and unused labels may move freely, while the relative order of the
larger blocks is preserved. These allowed exchanges connect all
assignments with the stated data.

The invariant polynomials have a basis consisting of the sums of
distinct words in each orbit. These are the finite-alphabet
$r$-dominant monomial functions of~\cite[Section 8.4]{ALW24}.
Setting variables outside $S$ to zero in an orbit sum on $Y$ leaves the
orbit sum with the same block data on $S$, or zero if there are too
many blocks. Every orbit sum on $S$ occurs in this way. This proves
the first equality over any field. Specialization is a surjective
graded algebra homomorphism, so it also gives the asserted equality
of the generated ideals.
\end{proof}

\subsection{Nilpotence of full-support words}
A word containing every variable at least $r$ times is fixed by every
simple transposition, and hence belongs to $J_n^{(r)}$. A product of
$r$ full-support words has this property. Thus the full-support ideal
$K_{n,n}$ satisfies
\begin{equation}\label{eq:r-terminal}
 K_{n,n}^{\,r}\subseteq J_n^{(r)}.
\end{equation}
This supplies the remaining hypothesis of Theorem~\ref{thm:general}.

\begin{theorem}\label{thm:r-bound}
For every $r\geq1$, the algebras $T_n^{(r)}$ are finite dimensional
over every field. More precisely, $T_1^{(r)}=\kk$, and for $n\geq2$,
\[
 \bigl((T_n^{(r)})_+\bigr)^{\,2r(r+1)^{n-2}}=0.
\]
In particular, $\deg\Hilb(T_n^{(r)};t)\leq2r(r+1)^{n-2}-1$.
\end{theorem}
\begin{proof}
For a singleton alphabet, every word is invariant, so $T_1^{(r)}=\kk$
and we may take $e_1=1$. Lemma~\ref{lem:r-restriction} gives compatibility
with restriction, and~\eqref{eq:r-terminal} permits $e_s=r$ for $s\geq2$.
Theorem~\ref{thm:general} therefore gives, for $n\geq2$,
\[
 b_n=e_n\prod_{s<n}(1+e_s)=2r(r+1)^{n-2}.
\]
\end{proof}

For $r=1$, this recovers the bound $2^{n-1}$ for $T_n$.
Sharper estimates on two and three variables improve
the bound for all larger alphabets.

\begin{corollary}\label{cor:r-char-zero}
For $n=2$, we have
\[
 T_2^{(r)}\cong\begin{cases}
 {\kk[t]/(t^2),} & \operatorname{char}(\kk)\neq 2,\\
 {\kk[t]/(t^{2r}),} & \operatorname{char}(\kk)=2.
 \end{cases}
\]
For $n\geq3$, for any field $\kk$,
\[
 (T_n^{(r)})_d=0,\qquad d\geq4r(r+1)^{n-3}.
\]
\end{corollary}
\begin{proof}
When $\operatorname{char}(\kk)\neq2$, the symmetric polynomials
$x_1+x_2$ and $x_1^2+x_2^2$ belong to $H_2^{(r)}$. In the quotient
they give $x_2=-x_1$ and $2x_1^2=0$. The map $x_1\mapsto t$,
$x_2\mapsto-t$ yields the stated isomorphism. Thus we may take
$b_1=1$ and $b_2=2$. Using~\eqref{eq:support-recurrence} with
$e_n=r$ for $n\geq3$ gives $b_3=4r$ and
$b_{n+1}=(r+1)b_n$ for $n\geq3$, hence the result.

When $\operatorname{char}(\kk)=2$, the symmetric polynomial $x_1+x_2$
gives $x_1=x_2$ in the quotient. If a nonempty word $w$ contains fewer
than $r$ occurrences of $x_1$ or of $x_2$, its $\star_r$-orbit consists
of $w$ and $s_1\star_r w$, whose sum vanishes upon setting $x_1=x_2$.
Otherwise, $w$ itself belongs to $H_2^{(r)}$ and has length at least
$2r$. Conversely, $x_1^rx_2^r\in H_2^{(r)}$. Therefore, the map
$x_1\mapsto t$, $x_2\mapsto t$ yields the stated isomorphism.

We may therefore take $b_1=1$ and $b_2=2r$. For $n=3$, the symmetric
relations give
\[
 \begin{cases}
 x_1+x_2+x_3=0,\\
 x_1^2+x_2^2+x_3^2=0,\\
 x_1^3+x_2^3+x_3^3=0
 \end{cases}
 \implies
 \begin{cases}
 x_3=x_1+x_2,\\
 x_1x_2=x_2x_1,\\
 x_1^2x_2=x_1x_2^2.
 \end{cases}
\]
In particular, $x_1^ix_2^{d-i}=x_1^{d-1}x_2$ for $d\geq2$ and
$1\leq i\leq d-1$. By the $r$-quasisymmetric relations,
\[
 \begin{aligned}
 x_1^{2r}+x_2^{2r}+x_1^{2r-1}x_2
 &=x_1^rx_2^r+(x_1^r+x_2^r)(x_1+x_2)^r\\
 &=x_1^rx_2^r+x_1^rx_3^r+x_2^rx_3^r=0.
 \end{aligned}
\]
Multiplying by $x_1$ or $x_2$ gives $x_1^{2r+1}=x_2^{2r+1}=0$.
In degree $2r+2$, pure powers vanish and every mixed monomial reduces
to $x_1^{2r+1}x_2=0$, so we may take $b_3=2r+2$.
Since $2r+2\leq4r$, this proves the asserted bound for $n=3$.

Finally, $b_4=r(1+b_1+b_2+b_3)=4r(r+1)$, and the recurrence
$b_{n+1}=(r+1)b_n$ for $n\geq4$ yields the result.
\end{proof}

\section{Examples and lower bounds}\label{sec:sharp}

The small-rank computations and symmetric bounds in
Sections~\ref{ssec:small-rank} and~\ref{ssec:general-lower} are in
characteristic zero. Section~\ref{ssec:ordered-lower} gives lower
bounds over every field.

\subsection{The coinvariant algebras in small rank}\label{ssec:small-rank}
We computed the $t$-graded Frobenius characteristic
$\mathcal F_n(t)$ of $C_n$ for $n\leq4${, suppressing $\mathbf z$}.
These computations independently verify the Hilbert series
of~\cite[Section 4.1]{ABL10}.
\begin{center}
\begin{tabular}{clrr}
\toprule
$n$ & $\mathcal F_n$& $\dim C_n$ & $\nu(C_n)$\\
\midrule
1 & $s_{1}$ & 1 & 1\\
2 & $s_{2}+s_{11}t$ & 2 & 2\\
3 & $s_{3}+s_{21}t+(s_{21}+s_{111})t^2+(s_{21}+s_{111})t^3$ & 9 & 4\\
4 & $ \mathcal F_4$ & 946 & 10\\
\bottomrule
\end{tabular}
\end{center}
where
\begin{align*}
{ \mathcal F_4}=s_{4} &+ s_{31}t +(s_{31}+s_{22}+s_{211})t^2+(2s_{31}+2s_{22}+3s_{211}+s_{1111})t^3\\
	&+(5s_{31}+4s_{22}+7s_{211}+3s_{1111})t^4+(11s_{31}+9s_{22}+15s_{211}+6s_{1111})t^5\\
	&+(21s_{31}+18s_{22}+29s_{211}+11s_{1111})t^6+(32s_{31}+28s_{22}+46s_{211}+18s_{1111})t^7\\
	&+(22s_{31}+22s_{22}+40s_{211}+18s_{1111})t^8 + {4s_{31}t^9.}\\
\end{align*}
The top-degree component of $C_4$ has Frobenius characteristic
$4s_{31}$ and contains no copy of the alternating representation.
Thus the classical description of the top degree by the sign
representation does not persist.

For five variables, we independently verify the previously known
Hilbert coefficients through degree $6$ and extend the computation
through degree $8$, obtaining
\[
 \Hilb(C_5;t)=1+4t+15t^2+55t^3+199t^4+712t^5
                +2520t^6+8819t^7+30490t^8+O(t^9).
\]

The exact dimension and vanishing degree of $C_5$ remain unknown.
The next subsection gives lower bounds without computing its full
relation matrices.

For the quasisymmetric quotients, our computations give
\begin{center}
\begin{tabular}{clrr}
\toprule
$n$ & $\Hilb(T_n;t)$ & $\dim T_n$ & $\nu(T_n)$\\
\midrule
1 & $1$ & 1 & 1\\
2 & $1+t$ & 2 & 2\\
3 & $1+2t+2t^2$ & 5 & 3\\
4 & $1+3t+7t^2+9t^3$ & 20 & 4\\
5 & $1+4t+14t^2+42t^3+88t^4+28t^5$ & 177 & 6\\
\bottomrule
\end{tabular}
\end{center}
These examples leave a substantial
gap between the exact vanishing degrees and the general upper bounds.

For the standardized coinvariants, we obtain the following complete
Hilbert polynomials:
\begin{center}
\begin{tabular}{clrr}
\toprule
$n$ & $\Hilb(\SC_n;t)$ & $\dim\SC_n$ & $\nu(\SC_n)$\\
\midrule
1 & $1$ & 1 & 1\\
2 & $1+t$ & 2 & 2\\
3 & $1+2t+3t^2+2t^3$ & 8 & 4\\
4 & $1+3t+8t^2+19t^3+36t^4+49t^5+40t^6+13t^7$ & 169 & 8\\
\bottomrule
\end{tabular}
\end{center}
For five variables, the initial coefficients are
\[
 \Hilb(\SC_5;t)=1+4t+15t^2+54t^3+186t^4+598t^5
                    +1736t^6+4360t^7+O(t^8).
\]
The full Hilbert polynomial of $\SC_5$ remains to be determined.

\subsection{General lower bounds for symmetric coinvariants}\label{ssec:general-lower}

Finiteness leaves open how rapidly these algebras grow with the
alphabet. A small auxiliary presentation gives explicit lower bounds
in every rank and a uniform asymptotic estimate.
Put $p_k=\sum_{i=1}^n x_i^k$, let
$V_n=(R_n)_1/\kk p_1$ be the standard representation, and set
\[
 B_n=T(V_n)/\inv{p_2,\ldots,p_{n-1}},\qquad
 s_n(t)=\sum_{k=2}^{n-1}t^k.
\]
Here the $p_k$ denote their images after eliminating $p_1$.

\begin{proposition}\label{prop:general-lower}
In Characteristic zero and for $n\geq3$, the  Hilbert series of the invariant subalgebra $B_n^{S_n}$ is
\begin{equation}\label{eq:general-invariant-series}
 \mathcal H_n(t):=\sum_{d\geq0}\dim(B_n)_d^{S_n}t^d
 =\frac1{n!}\sum_{\sigma\in S_n}
   \frac1{1-(\operatorname{fix}(\sigma)-1)t+s_n(t)},
\end{equation}
where $\operatorname{fix}(\sigma)$ is the number of fixed points of $\sigma$.
Define
\[
 F_n(t)=\mathcal H_n(t)-(n-1)t+s_n(t),\qquad
 F_n(t)^{-1}=\sum_{d\geq0}\ell_{n,d}t^d.
\]
If $\ell_{n,0},\ldots,\ell_{n,D}$ are nonnegative, then
\[
 \dim(C_n)_d\geq\ell_{n,d}{,\qquad0\leq d\leq D}.
\]
\end{proposition}
\begin{proof}
Work over $\mathbb Q$ and adjoin a primitive $n$th root of unity
$\zeta$. The Fourier coordinates
\[
 x_i=\sum_{j=1}^{n-1}\zeta^{(i-1)j}y_j
\]
identify the quotient by $p_1$ with the free algebra on
$y_1,\ldots,y_{n-1}$. In these coordinates,
\[
 p_k=n\!\!\sum_{\substack{j_1,\ldots,j_k\in[n-1]\\
                         j_1+\cdots+j_k\equiv0\pmod n}}
                    y_{j_1}\cdots y_{j_k}.
\]
For degree-lexicographic order with $y_1>\cdots>y_{n-1}$, the
leading word of $p_k$ is $y_1^{k-1}y_{n-k+1}$ for $2\leq k<n$.
These words have no inclusion or overlap ambiguities: their final
letters are distinct and different from $y_1$, whereas every proper
prefix consists only of $y_1$'s. The words avoiding these factors
therefore form a basis of $B_n$~\cite{Bergman78,Anick82}.

Let $E_n=\bigoplus_{k=2}^{n-1}\mathbb Qp_k$, with the indicated
degrees. Consider the sequence of graded left $B_n$-modules
\begin{equation}\label{eq:general-resolution}
 0\longrightarrow B_n\otimes E_n\xrightarrow{\delta}B_n\otimes V_n
 \xrightarrow{\mu}B_n\xrightarrow{\epsilon}\mathbb Q\longrightarrow0.
\end{equation}
Here $\mu(b\otimes v)=bv$ is multiplication and $\epsilon$ is the
augmentation. The map $\delta$ is obtained by right deconcatenation:
split off the last letter of each monomial in a relation in $T(V_n)$,
then project the remaining prefix to $B_n$. Explicitly,
\[
 \delta(b\otimes p_k)=\sum_{i=1}^n bx_i^{k-1}\otimes x_i,
 \qquad 2\leq k<n,
\]
where the $x_i$ denote the images of $x_i$ in $V_n$ and $B_n$ respectively.
The group $S_n$ acts on the left, diagonally on the tensor products:
\[
 \sigma\cdot(b\otimes v)=(\sigma\cdot b)\otimes(\sigma\cdot v),
 \qquad
 \sigma\cdot(b\otimes p_k)=(\sigma\cdot b)\otimes p_k.
\]
Its actions on $E_n$ and $\mathbb Q$ are trivial. All maps are defined
over $\mathbb Q$ and are $S_n$-equivariant.

We have $\mu\delta(b\otimes p_k)=bp_k=0$. The usual relation sequence
gives $\ker\mu=\operatorname{im}\delta$, and
$\operatorname{im}\mu=(B_n)_+=\ker\epsilon$.
To prove injectivity of $\delta$, use the Fourier coordinates above.
For a normal word $u$, the leading term of $\delta(u\otimes p_k)$
is a nonzero multiple of $uy_1^{k-1}\otimes y_{n-k+1}$.
The first factor remains normal and these terms are distinct,
proving injectivity of $\delta$ and hence exactness
of~\eqref{eq:general-resolution}.

For $\sigma\in S_n$ and a graded $S_n$-module $U$ with finite-dimensional
components, let
\[
 \operatorname{Tr}_U(\sigma;t)=\sum_{d\geq0}\tr(\sigma\big|_{U_d})t^d.
\]
Degree by degree, the Hopf trace formula~\cite[Chap. IV, Thm 23.2]{Bredon93}  on the
resolution~\eqref{eq:general-resolution} gives
\[
 \operatorname{Tr}_{B_n}(\sigma;t)
 \bigl(\operatorname{Tr}_{E_n}(\sigma;t)
       -\operatorname{Tr}_{V_n}(\sigma;t)+1\bigr)-1=0.
\]
Isolating $\operatorname{Tr}_{B_n}(\sigma;t)$ and averaging over $S_n$
gives~\eqref{eq:general-invariant-series} as
\[
 \begin{aligned}
 \mathcal H_n(t)
 &=\frac1{n!}\sum_{\sigma\in S_n}
   \frac1{1-\operatorname{Tr}_{V_n}(\sigma;t)+\operatorname{Tr}_{E_n}(\sigma;t)}
 =\frac1{n!}\sum_{\sigma\in S_n}
   \frac1{1-(\operatorname{fix}(\sigma)-1)t+s_n(t)}.
 \end{aligned}
\]

Since $p_1$ is invariant, Reynolds averaging lifts every invariant in
$T(V_n)=R_n/\inv{p_1}$ to an invariant in $R_n$. Hence
\[
 C_n=T(V_n)/\inv{T(V_n)^{S_n}_+}.
\]
Choose a homogeneous basis of $(B_n^{S_n})_+$. Averaging a homogeneous
lift of each basis element to $T(V_n)$ gives an invariant lift of the
same element. These lifts, together with $p_2,\ldots,p_{n-1}$, generate
the full invariant ideal. Indeed, from any positive-degree homogeneous
invariant in $T(V_n)$, subtract a linear combination of the chosen lifts
with the same image in $B_n$. The difference belongs to the kernel
$\inv{p_2,\ldots,p_{n-1}}$ of $T(V_n)\to B_n$.
Thus we obtain a presentation of $C_n$ with $n-1$ generators of degree
one and, in each degree $j\geq1$, the following number of relations:
\[
 r_{n,j}=\dim(B_n)_j^{S_n}+[t^j]s_n(t)
        =[t^j](\mathcal H_n(t)+s_n(t)-1).
\]
The first term counts the chosen invariant lifts, and the second counts
the relations $p_2,\ldots,p_{n-1}$. Since
\[
 1-(n-1)t+\sum_{j\geq1}r_{n,j}t^j
 =\mathcal H_n(t)-(n-1)t+s_n(t)=F_n(t),
\]
the graded Golod--Shafarevich inequality~\cite[Theorem 2.1]{Ershov12}
gives
\begin{equation}\label{eq:general-GS}
 \Hilb(C_n;t)F_n(t)=1+Q_n(t),
\end{equation}
{coefficientwise,} for some $Q_n(t)=\sum_{j\geq1}q_{n,j}t^j$, with $q_{n,j}\geq0$.
Multiplying~\eqref{eq:general-GS} by $F_n(t)^{-1}$ and comparing
coefficients gives, for $0\leq d\leq D$,
\[
 \dim(C_n)_d=\ell_{n,d}
       +\sum_{j=1}^d q_{n,j}\ell_{n,d-j}\geq\ell_{n,d},
\]
because every $\ell_{n,d-j}$ in the sum lies in the assumed
nonnegative range. This proves the asserted lower bounds.
\end{proof}

The coefficients are obtained by the integer recurrence
\begin{equation}\label{eq:general-reciprocal-recurrence}
 \ell_{n,0}=1,\qquad
 \ell_{n,d}=(n-1)\ell_{n,d-1}
           -\sum_{j=1}^d r_{n,j}\ell_{n,d-j}.
\end{equation}
Grouping permutations by their number $f$ of fixed points, with
multiplicity $\binom nf d_{n-f}$ where $d_j$ counts derangements,
makes~\eqref{eq:general-invariant-series} explicit. The initial
positive coefficients give, for example,
\begin{center}
\begin{tabular}{crr}
\toprule
$n$ & lower bound for $\nu(C_n)$ & lower bound for $\dim C_n$\\
\midrule
6 & $40$ & $10^{25}$\\
7 & $103$ & $10^{75}$\\
8 & $285$ & $10^{234}$\\
9 & $828$ & $10^{737}$\\
10 & $2\,455$ & $10^{2323}$\\
\bottomrule
\end{tabular}
\end{center}
The dimension bounds are rounded down from the sums of these
coefficients. The first negative reciprocal coefficient ends this
calculation; it does not assert vanishing in that degree.

\begin{theorem}\label{thm:symmetric-growth}
As $n\to\infty$, in characteristic zero,
\[
 \nu(C_n)\geq(1-o(1))n^5\log n,\qquad
 \log\dim C_n\geq(1-o(1))n^5(\log n)^2.
\]
\end{theorem}
\begin{proof}
Write $q=n-1$, $N=n!$, and $P(t)=1-qt+s_n(t)$.
We first note a useful consequence of the recurrence~\eqref{eq:general-reciprocal-recurrence}. For $t_\ast>0$, if
\[
 F_{n,\leq D}(t_\ast):=1-qt_\ast+\sum_{j=1}^D r_{n,j}t_\ast^j<0,
\]
then $\ell_{n,d}\geq t_\ast^{-d}$ for $d\leq D$.
Indeed, induction on consecutive ratios gives
\[
 \frac{\ell_{n,d}}{\ell_{n,d-1}}
 \geq q-\sum_{j=1}^d r_{n,j}t_\ast^{j-1}
 \geq q-\sum_{j=1}^D r_{n,j}t_\ast^{j-1}>t_\ast^{-1}.
\]
Proposition~\ref{prop:general-lower} then gives
$\dim(C_n)_d\geq t_\ast^{-d}$ through degree $D$.

Assume $n\ge 4$. We have $q\ge 3$, and for $0<t\le q^{-1}$, we have $P(t)>0$. On the other hand, $P((q-1)^{-1})<0$.
Therefore,  the least positive root $\rho$ of $P$ satisfies
\[
 q^{-1}<\rho<(q-1)^{-1},\qquad \rho=q^{-1}+O(q^{-3}).
\]
{We have the factorization}
\[
 P(t)=(1-t/\rho)(1-U(t)),\qquad
 U(t)=\sum_{k=2}^{n-1}\sum_{j=1}^{k-1}\rho^{k-j}t^j.
\]
Hence $U(t)\geq0$ coefficientwise and $U(\rho)=O(q^{-2})$.
Consequently, if $P(t)^{-1}=\sum c_dt^d$, then
\begin{equation}\label{eq:auxiliary-coefficient-bound}
 0\leq c_d\leq C\rho^{-d},\qquad
 C=(1-U(\rho))^{-1}=1+O(q^{-2}).
\end{equation}

Separate the identity permutation in~\eqref{eq:general-invariant-series}:
\[
 \mathcal H_n(t)=\frac1{NP(t)}+R(t).
\]
We claim that, uniformly for $t=\rho(1+O(\rho^4))$,
\begin{equation}\label{eq:nonidentity-expansion}
 R(t)=1+t^4+O(t^5)+O(N^{-1}).
\end{equation}
For $\sigma\in S_n$, put $M=\operatorname{fix}(\sigma)$ and
$\lambda=\max(2,M)$. The trace recurrence from
\eqref{eq:general-resolution} gives
$|\tr(\sigma\big|_{(B_n)_j})|\leq\lambda^j$, since
\[
 \frac{|M-1|}{\lambda}+\sum_{k=2}^{n-1}\lambda^{-k}\leq1.
\]
For $\sigma\ne1$, $\lambda\leq n-2$, and in the stated range
$(1-\lambda t)^{-1}\leq1+4\lambda$ for large $n$.
Counting ordered tuples of fixed points shows that every fixed
moment of $M$ is bounded uniformly in $n$. Thus the tail of $R(t)$
starting in degree five is bounded in absolute value by
$t^5\mathbb E[\lambda^5(1+4\lambda)]=O(t^5)$.
Here \(\mathbb E\) denotes the uniform average over \(S_n\). This expectation is bounded by a constant independent of \(n\).
The first four moments of $M-1$ are
$0,1,1,4$, giving $[t^j]\mathcal H_n=0$ for $1\leq j\leq3$ and
$[t^4]\mathcal H_n=1$ when $n\geq5$. Removing the identity changes the
evaluation of these initial terms by $O(N^{-1})$.
This proves~\eqref{eq:nonidentity-expansion}.

The degree-four term in~\eqref{eq:nonidentity-expansion} explains the scale of the bound. Taking
$t_\ast/\rho-1$ slightly larger than $\rho^4$ makes the negative
value of $P(t_\ast)$ outweigh the correction $R(t_\ast)-1$.
The identity contribution carries a factor $1/N$ and is bounded
by a geometric sum with ratio $t_\ast/\rho$. This allows a
truncation degree of order $\rho^{-4}\log N\sim n^5\log n$.
The following choice makes this balance precise.

Set
\[
 \varepsilon=q^{-1/2},\qquad
 \delta=(1+\varepsilon)\rho^4,\qquad
 t_\ast=\rho(1+\delta),\qquad
 D=\left\lfloor\frac{(1-\varepsilon)\log N}{\delta}\right\rfloor.
\]
Since $\rho P'(\rho)=-1+O(\rho^2)$,
\eqref{eq:nonidentity-expansion} gives
\[
 P(t_\ast)+R(t_\ast)-1=-\varepsilon\rho^4+O(\rho^5).
\]
For large $n$, $D\geq n-1$, and
\[
 F_{n,\leq D}(t_\ast)=P(t_\ast)+R(t_\ast)-1
 +\frac1N\sum_{j=0}^D c_jt_\ast^j
 -\sum_{j>D}[t^j]R(t)\,t_\ast^j.
\]
By~\eqref{eq:auxiliary-coefficient-bound}, the truncated contribution
of the identity permutation is at most
\[
 \frac1N\sum_{j=0}^D c_jt_\ast^j
 \leq\frac{C(1+\delta)^{D+1}}{N\delta}
 =O(N^{-\varepsilon}/\delta)=o(\varepsilon\rho^4).
\]
The tail of $R$ after degree $D$ is at most
$2n\exp(-D/(2n))$, because $(n-2)t_\ast\leq1-1/(2n)$ for large $n$.
It follows that $F_{n,\leq D}(t_\ast)<0$.
Hence $\nu(C_n)\geq D+1$ and
$\log\dim C_n\geq D\log(1/t_\ast)$.
Finally,
\[
 D\sim n^4\log(n!)\sim n^5\log n,\qquad
 \log(1/t_\ast)\sim\log n,
\]
which proves both estimates.
\end{proof}

This proves Theorem~\ref{Thm:GS}.

\begin{remark}[Five variables]\label{rem:C5-lower-bound}
The known values $\nu(C_1),\ldots,\nu(C_4)$ are $1,2,4,10$,
which agree with the numbers of standard Young tableaux of size
$n$, summed over all shapes. The next tableau number, for $n=5$,
is $26$. Although four matching terms are only suggestive, this
coincidence makes determining $\nu(C_5)$ exactly particularly
intriguing.

The known Hilbert coefficients sharpen Proposition~\ref{prop:general-lower}
to
\[
 \dim(C_5)_{17}\geq74\,845\,761,\qquad
 \nu(C_5)\geq18,\qquad \dim C_5\geq352\,171\,573.
\]
Indeed, writing $\ell_d=\ell_{5,d}$, the recurrence gives positive
coefficients through degree $17$. The Hilbert coefficients in
Section~\ref{ssec:small-rank} give
$Q_5=t^6+12t^7+87t^8+O(t^9)$ in~\eqref{eq:general-GS}.
Its remaining coefficients are nonnegative, so
\[
 \dim(C_5)_d\geq \widehat{\ell}_d
 :=\ell_d+\ell_{d-6}+12\ell_{d-7}+87\ell_{d-8},
 \qquad 0\leq d\leq17,
\]
where $\ell_j=0$ for $j<0$. Evaluating $\widehat{\ell}_{17}$ and
$\sum_{d=0}^{17}\widehat{\ell}_d$ gives the stated bounds. The next reciprocal
coefficient is $\ell_{18}=-1\,596\,163\,074$, so this calculation
does not establish nonvanishing in degree $18$.
\end{remark}

\subsection{Quasisymmetric and standardized coinvariants}\label{ssec:ordered-lower}

The Golod--Shafarevich argument also gives lower bounds for $T_n$
and $\SC_n$, using their defining word sums directly. Let
$H_n=\NCQSym_n$ or $\FQSym_n$, respectively, and put
$h_{n,d}=\dim(H_n)_d$, with $h_{n,0}=1$. Explicitly, for $d\geq1$,
\begin{equation}\label{eq:ordered-relation-dimensions}
 h^T_{n,d}=\sum_{k=1}^{\min(n,d)}k!\,\genfrac{\{}{\}}{0pt}{}{d}{k},
 \qquad
 h^{\SC}_{n,d}=\sum_{k=0}^{n-1}\genfrac{\langle}{\rangle}{0pt}{}{d}{k}.
\end{equation}
Here $\genfrac{\{}{\}}{0pt}{}{d}{k}$ is a Stirling number of the second kind, and
$\genfrac{\langle}{\rangle}{0pt}{}{d}{k}$ counts permutations of $[d]$
with $k$ descents. The second formula follows because
$F_\sigma(X_n)\ne0$ precisely when $\sigma$ has fewer than $n$
descents. Both formulas hold over every field, since the defining
basis elements have disjoint word supports.

\begin{theorem}\label{thm:ordered-growth}
Let $A_n=T_n$ or $\SC_n$, over any field, and use the corresponding
sequence in~\eqref{eq:ordered-relation-dimensions}. Put
\[
 r_{n,d}=h_{n,d}-2h_{n,d-1}+h_{n,d-2}\quad(d\geq2),
 \qquad
 \left(1-(n-1)t+\sum_{d\geq2}r_{n,d}t^d\right)^{-1}
 =\sum_{d\geq0}\ell_{n,d}t^d.
\]
If $\ell_{n,0},\ldots,\ell_{n,D}$ are nonnegative, then
$\dim(A_n)_d\geq\ell_{n,d}$ for $0\leq d\leq D$.
Moreover, as $n\to\infty$,
\begin{align*}
 \nu(T_n)&\geq n\log n-n\log\log n-O(n),\\
 \nu(\SC_n)&\geq n\log n-O(n),
\end{align*}
and, for both families,
\[
 \log\dim A_n\geq(1-o(1))n(\log n)^2.
\]
\end{theorem}
\begin{proof}
Write $H=H_n$ and $p=p_1$. We first remove redundant relations:
\[
 (pH_{d-1})\cap(H_{d-1}p)=pH_{d-2}p.
\]
Indeed, regarding $p$ as a free generator, $pf=gp$ forces
$f=up$ and $g=pu$. If $pu\in H$, then $u\in H$: compare the
coefficients of words $w,w'$ in the same packing or standardization
class, after prepending their respective least letters. The extended
words again have the same packing or standardization, and their
coefficients in $pu$ are those of $w,w'$ in $u$.
Choosing complements to $pH_{d-1}+H_{d-1}p$ and eliminating $p$
therefore gives a presentation on $n-1$ generators with at most
$r_{n,d}$ relations in degree $d$. The coefficient argument
of Proposition~\ref{prop:general-lower} proves the first assertion.

For the growth estimates, we use the truncated-denominator argument
from the proof of Theorem~\ref{thm:symmetric-growth}. Set
$t_\ast=(1+1/n)/(n-1)$.
For $d\geq n$, the nonnegative series $(e^z-1)^k$, evaluated at
$z=d/k$, and Stirling's estimate give
\[
 k!\,\genfrac{\{}{\}}{0pt}{}{d}{k}\leq
 C\sqrt d\,k^d\exp(-k e^{-d/k}).
\]
The right-hand side increases with $k\leq n$ when $d\geq n$.
Hence
\begin{equation}\label{eq:ordered-count-estimates}
 h^T_{n,d}\leq Cn\sqrt d\,n^d e^{-n e^{-d/n}},
 \qquad
 h^{\SC}_{n,d}\leq C\sqrt d\,n^d e^{-d e^{-d/n}}.
\end{equation}
For the second bound, inclusion-exclusion identifies
$h^{\SC}_{n,d}/d!$ with the volume of
$\{u\in[0,1]^d:\sum u_i\leq n\}$. Bounding its indicator by
$\exp((d/n)(n-\sum u_i))$ and integrating proves the estimate.

Take
\[
 D_T=\lfloor n(\log n-\log\log n-\log8)\rfloor,\qquad
 D_{\SC}=\lfloor n(\log n-\log8)\rfloor.
\]
For $n\leq d\leq D_T$, we have $ne^{-d/n}\geq8\log n$.
For $n\leq d\leq D_{\SC}$, the function $de^{-d/n}$ is decreasing,
so $de^{-d/n}\geq8(\log n-\log8)$.
Thus the exponential factors in~\eqref{eq:ordered-count-estimates}
are at most $n^{-8}$ and $8^8n^{-8}$, respectively.
Also, $(nt_\ast)^d\leq e^{3d/n}\leq n^3$ on both intervals for
large $n$. Since $D_A=O(n\log n)$ for $A=T,\SC$, summing the
pointwise estimates gives
\[
 \sum_{d=n}^{D_T}h^T_{n,d}t_\ast^d
 =O\!\left(n^{-5/2}(\log n)^{3/2}\right),\qquad
 \sum_{d=n}^{D_{\SC}}h^{\SC}_{n,d}t_\ast^d
 =O\!\left(n^{-7/2}(\log n)^{3/2}\right).
\]
Both sums are $o(n^{-2})$.
The degrees $2\leq d<n$ contribute $O(n^{-2})$: use
$h^{\SC}_{n,d}\leq h^T_{n,d}\leq2^{d-1}d!$ and split the sum at
$\lfloor n/4\rfloor$. The initial terms of this majorant have consecutive ratio
$2(d+1)t_\ast\leq3/4$ for large $n$, so their sum is $O(n^{-2})$.
Stirling's estimate makes both endpoint values of the remaining
interval exponentially small, and log-convexity gives an exponentially
small remaining sum. Since $0\leq r_{n,d}\leq h_{n,d}$, the truncated
denominator satisfies
\[
 1-(n-1)t_\ast+\sum_{d=2}^{D_A}r_{n,d}t_\ast^d
 =-1/n+O(n^{-2})<0.
\]
The same ratio induction as before gives
$\dim(A_n)_d\geq t_\ast^{-d}$ through degree $D_A$.
Thus $\nu(A_n)\geq D_A+1$ and
$\log\dim A_n\geq D_A\log(1/t_\ast)$, proving the estimates.
\end{proof}

In particular, $\nu(T_n)/n$ and $\nu(\SC_n)/n$ tend to infinity,
and both dimensions eventually exceed $(n!)^c$ for every fixed
$c>0$. The small-rank values alone do not reveal this growth.
The reciprocal recurrence~\eqref{eq:general-reciprocal-recurrence},
with $r_{n,1}=0$, gives the following explicit lower bounds:
\begin{center}
\begin{tabular}{crrrr}
\toprule
 & \multicolumn{2}{c}{$T_n$} & \multicolumn{2}{c}{$\SC_n$}\\
$n$ & $\nu(T_n)\geq$ & $\dim T_n\geq$
    & $\nu(\SC_n)\geq$ & $\dim\SC_n\geq$\\
\midrule
8 & $9$ & $10^6$ & $15$ & $10^{10}$\\
10 & $13$ & $10^{10}$ & $21$ & $10^{18}$\\
20 & $36$ & $10^{43}$ & $54$ & $10^{67}$\\
50 & $115$ & $10^{191}$ & $171$ & $10^{286}$\\
100 & $273$ & $10^{541}$ & $400$ & $10^{794}$\\
\bottomrule
\end{tabular}
\end{center}
The dimension bounds are rounded down from sums of certified
coefficients, as in Section~\ref{ssec:general-lower}.

\section{Adding noncommuting alphabets}\label{sec:diagonal}

For freely noncommuting variables, the passage to several alphabets
admits an exact description. The one-alphabet coinvariants determine
the multigraded dimensions and, in the symmetric case, the entire
multigraded character.

For $m\geq1$, let
\[
 R_{n,m}=\kk\inv{x_i^{(a)}:i\in[n],\ a\in[m]},
\]
where all letters are freely noncommuting, and give $x_i^{(a)}$
multidegree $\mathbf e_a$. The symmetric group {$S_n$} acts diagonally by
\[
 \sigma\cdot x_i^{(a)}=x_{\sigma(i)}^{(a)},
\]
for $\sigma\in S_n$. For a word $w$, define its index support by
\[
 \supp_{\mathrm{ind}}(w)=\{i\in[n]:x_i^{(a)}\text{ occurs in }w
       \text{ for some }a\in[m]\}.
\]
The diagonal quasi-symmetrizing action applies the rule of
Remark~\ref{rem:Hivert} to these indices:
$s_i\star w=s_i\cdot w$ if at least one of $i,i+1$ is absent from
$\supp_{\mathrm{ind}}(w)$, and $s_i\star w=w$ otherwise.
Denote the respective invariant
subalgebras by
\[
 \mathrm{NCDSym}_{n,m}=R_{n,m}^{S_n},
 \qquad
 \mathrm{NCDQSym}_{n,m}=R_{n,m}^{\star S_n},
\]
and put
\begin{align*}
 I_{n,m}&=\inv{\mathrm{NCDSym}_{n,m}^{+}},
 &C_{n,m}&=R_{n,m}/I_{n,m},\\
 J_{n,m}&=\inv{\mathrm{NCDQSym}_{n,m}^{+}},
 &T_{n,m}&=R_{n,m}/J_{n,m}.
\end{align*}
Thus $C_{n,1}=C_n$ and $T_{n,1}=T_n$. Diagonal realizations in
noncommuting alphabets also appear in~\cite[Section 7]{BL08}.

For $\mathbf a=(a_1,\ldots,a_d)\in[m]^d$, define the coloring map
\[
 c_{\mathbf a}(x_{i_1}\cdots x_{i_d})
   =x_{i_1}^{(a_1)}\cdots x_{i_d}^{(a_d)}.
\]
For standardized coinvariants, standardize the index word while
keeping the alphabet-label sequence fixed. Thus put
\[
 \SI_{n,m}
 =\inv{c_{\mathbf a}(F_\sigma(X_n)):
       d\geq1,\ \sigma\in S_d,\ \mathbf a\in[m]^d},
 \qquad
 \SC_{n,m}=R_{n,m}/\SI_{n,m}.
\]
In particular, $\SC_{n,1}=\SC_n$.

\begin{theorem}\label{thm:alphabet-independence}
Let $U=\kk^m$ and $D=C,T$, or $\SC$. Over any field, for every
$n,m\geq1$ and $d\geq0$,
\[
 (D_{n,m})_d\cong(D_n)_d\otimes U^{\otimes d}.
\]
The direct sums over $d$ are graded algebra isomorphisms,
with multiplication on the right given by
\[
 (a\otimes u)(b\otimes v)=ab\otimes(u\otimes v)
\]
for homogeneous $a,b$ and tensors $u,v$ of the corresponding degrees.
For $D=C$, the isomorphism is $S_n$-equivariant, with trivial $S_n$-action
on $U$. All three are equivariant for the natural $\mathrm{GL}(U)$-action
by linear substitutions among alphabets.

Consequently, for $D=C,T$, or $\SC$,
\[
 \Hilb(D_{n,m};q_1,\ldots,q_m)
 =\Hilb(D_n;q_1+\cdots+q_m),
 \qquad
 \Hilb(D_{n,m};t)=\Hilb(D_n;mt).
\]
In characteristic zero, with $\mathcal F_n$ as in the introduction,
\begin{align*}
 \operatorname{Frob}_{S_n}(C_{n,m};\mathbf z;q_1,\ldots,q_m)
 &=\mathcal F_n(\mathbf z;q_1+\cdots+q_m),\\
 \operatorname{Frob}_{S_n}(C_{n,m};\mathbf z;t)
 &=\mathcal F_n(\mathbf z;mt).
\end{align*}
Here $q_a$ records the degree in the $a$th alphabet, and $t$ records
total degree.
\end{theorem}

\begin{proof}
The coloring maps decompose $(R_{n,m})_d$ into copies of $(R_n)_d$,
indexed by alphabet-label sequences. Each map intertwines both
the ordinary and the quasi-symmetrizing actions with their
one-alphabet counterparts. Their invariant subspaces therefore
decompose in the same way. The standardized relations are colorings
of their one-alphabet counterparts by definition.

Coloring also respects concatenation:
\[
 c_{\boldsymbol\alpha\boldsymbol\beta}(uv)
 =c_{\boldsymbol\alpha}(u)c_{\boldsymbol\beta}(v).
\]
To identify the generated ideals, decompose each factor of a product
$ufv$, where $f$ is a positive-degree invariant or a standardized
relation, by alphabet-label sequence. Every resulting product is a coloring of a one-alphabet
relation of the same form. Conversely, coloring a one-alphabet product
$ufv$ gives a product whose middle factor is a defining relation, so it belongs to
the corresponding multialphabet ideal. Thus
\begin{align*}
 (I_{n,m})_d
 &=\bigoplus_{\mathbf a\in[m]^d}c_{\mathbf a}((I_n)_d),\\
 (J_{n,m})_d
 &=\bigoplus_{\mathbf a\in[m]^d}c_{\mathbf a}((J_n)_d),\\
 (\SI_{n,m})_d
 &=\bigoplus_{\mathbf a\in[m]^d}c_{\mathbf a}((\SI_n)_d).
\end{align*}
Taking quotients proves the decompositions. The identification with
the standard tensor basis of $U^{\otimes d}$ proves the stated
equivariance.
Compatibility with multiplication follows from the concatenation
identity above.

The copy indexed by $\mathbf a$ has multigraded weight
$q_{a_1}\cdots q_{a_d}$. Since
\[
 \sum_{\mathbf a\in[m]^d}q_{a_1}\cdots q_{a_d}
 =(q_1+\cdots+q_m)^d,
\]
summing dimensions, or $S_n$-characters, proves the multigraded
identities. Specializing every $q_a$ to $t$ gives the total-degree
identities.
\end{proof}

For $D=C$ when $\operatorname{char}(\kk)\nmid n!$, and for
$D=T$ or $\SC$ over every field, these quotients are finite dimensional and
\[
 \nu(D_{n,m})=\nu(D_n),
 \qquad
 \dim_\kk D_{n,m}=\Hilb(D_n;m).
\]
All one-alphabet degree bounds therefore apply unchanged.

\begin{remark}
The symmetric factorization is a property of tensor algebras.
Let $D_G(V)$ be the quotient defined in Remark~\ref{rem:general-coinvariants}, and
let $U$ be a finite-dimensional vector space with trivial $G$-action.
The same proof gives
\[
 D_G(V\otimes U)_d\cong D_G(V)_d\otimes U^{\otimes d},
\]
equivariantly for $G\times\mathrm{GL}(U)$ and compatibly with the
multiplication described above. No characteristic assumption is needed.
\end{remark}

{
Assume now that $\kk$ has characteristic zero. The first vanishing
degree also controls stability as the number of alphabets increases.
Indeed, $s_1^d=\sum_{\lambda\vdash d}f^\lambda s_\lambda$, where
$f^\lambda$ is the number of standard tableaux of shape $\lambda$.
Theorem~\ref{thm:alphabet-independence} therefore gives stable Schur
multiplicities for $D_{n,m}$, with every constituent visible once
$m\geq\nu(D_n)-1$. This extends to mixed alphabets by projection.

Let $E,F,W$ range over finite-dimensional alphabet spaces, and let
$A_n(E,F,W)$ be a family of quotients of
$T(\kk^n\otimes(E\oplus F\oplus W))$ by positive-degree homogeneous
$S_n$-stable relations natural under all linear maps between alphabet
spaces of each type.
For example,
\[
 A_n(E,F,W)=\Sym(\kk^n\otimes E)\otimes
 \bigwedge(\kk^n\otimes F)\otimes T(\kk^n\otimes W)
\]
has commuting, anticommuting, and free letters, respectively, with
letters in distinct tensor factors commuting. For $D=C,T$, or $\SC$,
denote by $D_n^{\mathrm{mix}}(E,F,W)$ the quotient of $A_n(E,F,W)$
by the images of the corresponding colored defining relations above.
For $D=C$, this is the quotient by all positive-degree $S_n$-invariants:
averaging lifts every invariant through the projection.

\begin{corollary}\label{cor:mixed-stability}
In characteristic zero, fix $n$, let $D=C,T$, or $\SC$, and set $N=\nu(D_n)-1$.
There are finite-dimensional multiplicity spaces $M_{\lambda,\mu,\eta}$,
independent of $E,F,W$, such that
\[
 D_n^{\mathrm{mix}}(E,F,W)\cong
 \bigoplus_{|\lambda|+|\mu|+|\eta|\leq N}
 M_{\lambda,\mu,\eta}\otimes
 \mathbb S_\lambda(E)\otimes\mathbb S_\mu(F)\otimes\mathbb S_\eta(W),
\]
as graded $\mathrm{GL}(E)\times\mathrm{GL}(F)\times\mathrm{GL}(W)$-modules,
where $\mathbb S_\lambda$ denotes the Schur functor. For $D=C$, the
multiplicity spaces are $S_n$-modules and the isomorphism is also
$S_n$-equivariant. Thus the characters have universal nonnegative
Schur coefficients, with an additional Schur factor in $\mathbf z$
in the symmetric case. Every constituent is visible when each
alphabet space has dimension at least $N$; absent types may be
omitted. In particular,
$\nu(D_n^{\mathrm{mix}}(E,F,W))\leq\nu(D_n)$.
\end{corollary}

\begin{proof}
The mixed quotient is a quotient of $D_{n,m}$, where
$m=\dim E+\dim F+\dim W$. For $m=0$ the quotient is $\kk$;
otherwise Theorem~\ref{thm:alphabet-independence} gives the degree bound.
Polynomial representations of the three
general linear groups decompose into products of Schur functors.
Adding alphabets and setting them to zero preserve the defining
relations; their composite is the identity on the smaller quotient.
The part of degree zero in the extra alphabets is therefore precisely
the smaller quotient. Its character is obtained by setting the extra
variables to zero. Since a Schur polynomial specializes to
the same Schur polynomial or to zero, its multiplicity is independent
of the alphabet dimensions whenever it is visible. All partitions
involved have size at most $N$, proving the stated stable range.
The same argument applies to each $S_n$-isotypic component when $D=C$.
\end{proof}

For commuting symmetric coinvariants, this stability was proved by
F.~Bergeron~\cite[Theorem 2.2]{Ber13}. The bosonic--fermionic case is
discussed in~\cite[Section 2]{Ber20BF} and strengthened to a universal
super-Schur expansion by Lentfer~\cite[Theorem 1.2]{Lentfer25}.
Related quasisymmetric harmonic spaces occur in~\cite{ABB05,BCSZ23}.
Our method treats all three families and mixtures containing free
alphabets uniformly: the one-alphabet bound $\nu(D_n)$ controls both
their degrees and a common sufficient range for representation stability.

\begin{remark}
For $D=C,T$, or $\SC$, requiring distinct alphabets to commute, without
imposing commutation within an alphabet, gives a further quotient
$\overline D_{n,m}$ of $D_{n,m}$. Projection and the inclusion of one
alphabet give $\nu(\overline D_{n,m})=\nu(D_n)$ for $m\geq1$.
Permuting alphabets and setting extra alphabets to zero preserve the
relations, so the multigraded Hilbert series are specializations of a
single symmetric function of degree at most $\nu(D_n)-1$. The same
holds for the $S_n$-Frobenius series when $D=C$.
Arbitrary linear substitutions among alphabets, however, need not
preserve these commutation relations: replacing every $x_j^{(2)}$
by $x_j^{(2)}+x_j^{(1)}$ sends $[x_i^{(1)},x_j^{(2)}]=0$ to a relation
requiring $[x_i^{(1)},x_j^{(1)}]=0$. Thus the series stabilize, but
the interpretation of their Schur coefficients as
$\mathrm{GL}_m$-multiplicities in Corollary~\ref{cor:mixed-stability}
does not follow without its compatibility hypothesis.
\end{remark}

\begin{remark}
Taking graded duals of the projection embeds the corresponding inverse
system into the free inverse system, with contragredient actions, as
in~\cite[Corollary 2.5]{ABL10}. The usual polynomial harmonic realization
$H_{n,m}$ of the symmetric commuting quotient has the same graded
$\mathrm{GL}_m\times S_n$-character as that quotient, and hence the same
stability. The Vandermonde spaces $\mathcal E_{n,m}\subseteq H_{n,m}$
generated by differentiation and polarization are preserved by alphabet
inclusions and by setting extra alphabets to zero~\cite[Section 2]{Ber20}.
Thus the specialization argument above also gives their stability,
without assuming $\mathcal E_{n,m}=H_{n,m}$.
\end{remark}

\begin{remark}
The same results extend to the $r$-quasisymmetric coinvariants
$T_n^{(r)}$ of Section~\ref{sec:r-ncqsym}, for every fixed $r\geq1$.
For the diagonal $r$-quasi-symmetrizing action, count occurrences of
each index across all alphabets. Coloring preserves these counts,
so the proof of Theorem~\ref{thm:alphabet-independence} gives the same
tensor decomposition, Hilbert-series identities, and equality of first
vanishing degrees over every field. In characteristic zero, the same
projection and specialization arguments give stable Schur multiplicities
for the corresponding mixed-alphabet quotients under the hypotheses of
Corollary~\ref{cor:mixed-stability}, with $\nu(T_n^{(r)})-1$ as both
the degree bound and a sufficient stable dimension for each alphabet
space present.
\end{remark}
}

\section{Further questions}\label{sec:remaining}

{For freely noncommuting alphabets, the results above reduce the
multidiagonal dimension and character problems to one alphabet.}
We conclude with four questions about
these finite-dimensional algebras and the scope of the support method.
Throughout this section, the field has characteristic zero.

\begin{question}
What are the exact Hilbert polynomials of $C_n$, $T_n$, and
$\SC_n$, and their first vanishing degrees?
\end{question}

A combinatorial basis or a description of normal words would address
both parts at once. The small-rank polynomials in
Section~\ref{sec:sharp} give initial data, and
Remark~\ref{rem:C5-lower-bound} shows that $C_5$ survives
through degree $17$. The growth estimates in
Theorems~\ref{thm:symmetric-growth} and~\ref{thm:ordered-growth}
leave open the asymptotic orders of the dimensions and vanishing
degrees.
For $T_n$, Remark~\ref{rem:ncqsym-improvement}
shows how the explicit quasisymmetric relations improve on
full-support vanishing.
 For the
intermediate quotients $T_n^{(r)}$, a filtration recording letter
multiplicities as well as support may improve
Theorem~\ref{thm:r-bound}.

\begin{question}
Can the Schur coefficients of $\mathcal F_n(\mathbf z;t)$
be described uniformly in $n$?
\end{question}

Section~\ref{ssec:small-rank} determines them for $n\leq4$.
Writing
\[
 \mathcal F_n(\mathbf z;t)
 =\sum_{d\geq0}\sum_{\lambda\vdash n}
   m_{\lambda,d}t^d s_\lambda(\mathbf z),
\]
the coefficient $m_{\lambda,d}$ is the multiplicity of the Specht
module $S^\lambda$ in $(C_n)_d$. Averaging shows that the trivial
representation occurs only in degree zero. For $n\geq2$, the sole
linear relation is $x_1+\cdots+x_n=0$, so
$(C_n)_1\cong S^{(n-1,1)}$. A description of the higher-degree
multiplicities would, by Theorem~\ref{thm:alphabet-independence},
determine the multigraded characters for every number of alphabets.

\begin{question}\label{q:sharp}
What is the optimal upper bound for $\nu(R_X/L_X)$ among families
$(L_S)$ on the nonempty subsets of an $n$-letter alphabet that are
compatible with restriction and contain every full-support word?
In particular, is the bound $7\cdot2^{n-4}$ of
Proposition~\ref{prop:seven-bound} attained for every $n\geq4$?
\end{question}

In characteristic zero, the full abelianization preimages of the
ordinary symmetric coinvariant ideals form another admissible family:
the product of all variables is an elementary symmetric polynomial,
so every full-support word lies in the corresponding ideal. The
four-variable quotient has highest degree $6$, attaining the new
bound $\nu=7$. Sharpness for larger alphabets, and simultaneous
attainment by a single family on all finite alphabets, remain open.
These problems concern the general restriction criterion, beyond
the particular relations defining $T_n$.

\begin{question}
Do $C_n$ and $T_n$ admit natural geometric realizations, for example
as convolution algebras, compatible with their abelianization maps
to commutative coinvariant rings?
\end{question}
Such realizations could explain their Hilbert polynomials and, for
$C_n$, their graded $S_n$-characters.

\bibliographystyle{alpha}
\bibliography{references}

\newcommand{\etalchar}[1]{$^{#1}$}
\begin{thebibliography}{BGN{\etalchar{+}}26}

\bibitem[AB03]{AB03}
Jean-Christophe Aval and Nantel Bergeron.
\newblock {Catalan} paths and quasi-symmetric functions.
\newblock {\em Proceedings of the American Mathematical Society},
  131(4):1053--1062, 2003.

\bibitem[ABB04]{ABB04}
Jean-Christophe Aval, Fran{\c{c}}ois Bergeron, and Nantel Bergeron.
\newblock Ideals of quasi-symmetric functions and super-covariant polynomials
  for {$S_n$}.
\newblock {\em Advances in Mathematics}, 181(2):353--367, 2004.

\bibitem[ABB06]{ABB05}
Jean-Christophe Aval, Fran{\c{c}}ois Bergeron, and Nantel Bergeron.
\newblock Diagonal {Temperley--Lieb} invariants and harmonics.
\newblock {\em S{\'e}minaire Lotharingien de Combinatoire}, 54A:Art. B54Aq, 19
  pp., 2006.

\bibitem[ABL10]{ABL10}
Jean-Christophe Aval, Nantel Bergeron, and Huilan Li.
\newblock Non-commutative combinatorial inverse systems.
\newblock {\em International Journal of Algebra}, 4(21):1003--1020, 2010.

\bibitem[ALvW24]{ALW24}
Farid Aliniaeifard, Shu~Xiao Li, and Stephanie van Willigenburg.
\newblock Generalized chromatic functions.
\newblock {\em International Mathematics Research Notices}, 2024(5):4456--4500,
  2024.

\bibitem[Ani82]{Anick82}
David~J. Anick.
\newblock Non-commutative graded algebras and their {Hilbert} series.
\newblock {\em Journal of Algebra}, 78(1):120--140, 1982.

\bibitem[BCSZ23]{BCSZ23}
Nantel Bergeron, Kelvin Chan, Farhad Soltani, and Mike Zabrocki.
\newblock Quasisymmetric harmonics of the exterior algebra.
\newblock {\em Canadian Mathematical Bulletin}, 66(3):997--1013, 2023.

\bibitem[Ber78]{Bergman78}
George~M. Bergman.
\newblock The diamond lemma for ring theory.
\newblock {\em Advances in Mathematics}, 29:178--218, 1978.

\bibitem[Ber13]{Ber13}
Fran{\c{c}}ois Bergeron.
\newblock Multivariate diagonal coinvariant spaces for complex reflection
  groups.
\newblock {\em Advances in Mathematics}, 239:97--108, 2013.

\bibitem[Ber20a]{Ber20BF}
Fran{\c{c}}ois Bergeron.
\newblock The bosonic-fermionic diagonal coinvariant modules conjecture.
\newblock arXiv preprint 2005.00924, 2020.

\bibitem[Ber20b]{Ber20}
Fran{\c{c}}ois Bergeron.
\newblock {$(GL_k\times S_n)$}-modules of multivariate diagonal harmonics.
\newblock arXiv preprint 2003.07402, 2020.

\bibitem[BG21]{BG21}
Czes{\l}aw Bagi{\'n}ski and Piotr Grzeszczuk.
\newblock On the generic family of {Cayley} graphs of a finite group.
\newblock {\em Journal of Combinatorial Theory, Series A}, 184:105495, 2021.

\bibitem[BGN{\etalchar{+}}26]{BGNST26}
Nantel Bergeron, Lucas Gagnon, Philippe Nadeau, Hunter Spink, and Vasu Tewari.
\newblock The quasisymmetric flag variety: A toric complex on noncrossing
  partitions.
\newblock {\em Forum of Mathematics, Pi}, 14:e18, 2026.

\bibitem[BI73]{BI73}
George~M. Bergman and I.~Martin Isaacs.
\newblock Rings with fixed-point-free group actions.
\newblock {\em Proceedings of the London Mathematical Society}, 27(1):69--87,
  1973.

\bibitem[BL08]{BL08}
Fran{\c{c}}ois Bergeron and Aaron Lauve.
\newblock Invariant and coinvariant spaces for the algebra of symmetric
  polynomials in non-commuting variables.
\newblock {\em Discrete Mathematics \& Theoretical Computer Science
  Proceedings}, AJ:543--556, 2008.
\newblock {FPSAC} 2008.

\bibitem[Bor53]{Borel53}
Armand Borel.
\newblock Sur la cohomologie des espaces fibr{\'e}s principaux et des espaces
  homog{\`e}nes de groupes de {Lie} compacts.
\newblock {\em Annals of Mathematics}, 57(1):115--207, 1953.

\bibitem[Bre93]{Bredon93}
Glen~E. Bredon.
\newblock {\em Topology and Geometry}, volume 139 of {\em Graduate Texts in
  Mathematics}.
\newblock Springer-Verlag, New York, 1993.

\bibitem[BRRZ08]{BRRZ08}
Nantel Bergeron, Christophe Reutenauer, Mercedes Rosas, and Mike Zabrocki.
\newblock Invariants and coinvariants of the symmetric group in noncommuting
  variables.
\newblock {\em Canadian Journal of Mathematics}, 60(2):266--296, 2008.

\bibitem[BZ09]{BZ09}
Nantel Bergeron and Mike Zabrocki.
\newblock The {Hopf} algebras of symmetric functions and quasi-symmetric
  functions in non-commutative variables are free and co-free.
\newblock {\em Journal of Algebra and Its Applications}, 8(4):581--600, 2009.

\bibitem[Che55]{Chevalley55}
Claude Chevalley.
\newblock Invariants of finite groups generated by reflections.
\newblock {\em American Journal of Mathematics}, 77(4):778--782, 1955.

\bibitem[DHT02]{DHT02}
G{\'e}rard Duchamp, Florent Hivert, and Jean-Yves Thibon.
\newblock Noncommutative symmetric functions {VI}: Free quasi-symmetric
  functions and related algebras.
\newblock {\em International Journal of Algebra and Computation},
  12(5):671--717, 2002.

\bibitem[Ers12]{Ershov12}
Mikhail Ershov.
\newblock {Golod--Shafarevich} groups: a survey.
\newblock {\em International Journal of Algebra and Computation},
  22(5):1230001, 2012.

\bibitem[Hai02]{Haiman02}
Mark Haiman.
\newblock Vanishing theorems and character formulas for the {Hilbert} scheme of
  points in the plane.
\newblock {\em Inventiones Mathematicae}, 149(2):371--407, 2002.

\bibitem[Hiv99]{H99}
Florent Hivert.
\newblock {\em Combinatoire des fonctions quasi-symétriques}.
\newblock PhD thesis, Université de Marne-la-Vallée, 1999.
\newblock 1999MARN0051.

\bibitem[Hiv00]{Hivert00}
Florent Hivert.
\newblock {Hecke} algebras, difference operators, and quasi-symmetric
  functions.
\newblock {\em Advances in Mathematics}, 155(2):181--238, 2000.

\bibitem[Hiv05]{H05}
Florent Hivert.
\newblock Local action of the symmetric group and generalizations of
  quasi-symmetric functions.
\newblock Unpublished preprint, 27 pp., version dated December 4,
  https://perso.lisn.upsaclay.fr/hivert/PAPER/GenQSym.ps, 2005.

\bibitem[Len25]{Lentfer25}
John Lentfer.
\newblock Diagonal supersymmetry for coinvariant rings.
\newblock arXiv preprint 2505.14885, 2025.

\bibitem[MR95]{MR95}
Claudia Malvenuto and Christophe Reutenauer.
\newblock Duality between quasi-symmetric functions and the {Solomon} descent
  algebra.
\newblock {\em Journal of Algebra}, 177(3):967--982, 1995.

\bibitem[RS06]{RS06}
Mercedes~H. Rosas and Bruce~E. Sagan.
\newblock Symmetric functions in noncommuting variables.
\newblock {\em Transactions of the American Mathematical Society},
  358(1):215--232, 2006.

\end{thebibliography}
\end{document}